\documentclass[11pt,a4paper,reqno]{article}
\usepackage{color}

\usepackage{geometry}
\usepackage{color}
\usepackage{hyperref}
\usepackage[english]{babel}
\usepackage[utf8]{inputenc}
\usepackage{csquotes} 
\usepackage{amsmath,amssymb,amsthm,mathrsfs}
\usepackage{epsfig,color}
\usepackage{bbm,amsmath,amsthm,epsfig,latexsym,marvosym}
\usepackage{amsfonts}
\usepackage{bigints}
\usepackage{amsmath}
\usepackage{amssymb}
\usepackage{subfig}
\usepackage{esint}
\usepackage{enumerate}

\def\R{\mathbb R}
\def\N{\mathbb N}
\def\C{\mathcal C}

\def\cal{\mathcal}

\def\F{{\cal F}}
\def\H{{\cal H}}
\def\M{{\cal M}}
\def\L{{\cal L}}

\def\a{\alpha}
\def\b{\beta}
\def\g{\gamma_m}
\def\de{\delta}
\def\e{\varepsilon}
\def\l{\lambda}
\def\om{\omega}

\def\ov{\overline}

{\left\lbrace\begin{array}{@{}l@{}}}%
{\end{array}\right.}

\def\pa{\partial}

\def\d{\, \mathrm{d}}

\def\bd{{\rm bd}}

\def\ca{\mathbbmss{1}}

\def\pared{\partial^{*}}

\def\00{{\bf 0}}
\def\dive{{\rm div}}

\def\wt{\stackrel{*}{\rightharpoonup}}

\newcommand{\cc}{\subset\subset}

\newcommand{\restr}{%
  \,\raisebox{-.127ex}{\reflectbox{\rotatebox[origin=br]{-90}{$\lnot$}}}\,%
}

\newtheorem{theorem}{Theorem}[section]
\newtheorem{corollary}[theorem]{Corollary}
\newtheorem{proposition}[theorem]{Proposition}
\newtheorem{example}[theorem]{Example}
\newtheorem{lemma}[theorem]{Lemma}

\theoremstyle{definition}
\newtheorem{remark}[theorem]{Remark}
\newtheorem{definition}[theorem]{Definition}

\numberwithin{equation}{section}
\numberwithin{figure}{section}

\usepackage{authblk}

\begin{document}
  
\title{Equilibria configurations for \\epitaxial crystal growth with adatoms}
\author[1]{Marco Caroccia}
\author[2]{Riccardo Cristoferi}
\author[3]{Laurent Dietrich}

\affil[1]{Faculdade de Ci\^{e}ncias, Departamento de Matem\'{a}tica, Universidade de Lisboa, 1749-016 Lisboa, Portugal}
\affil[2]{Department of Mathematical Sciences \\ Carnegie Mellon University \\ Pittsburgh, PA 15213}
\affil[3]{Lycée Fabert, Bâtiment Toqueville (CPGE),  57000 Metz, France}

   \maketitle

\begin{abstract}
The behavior of a surface energy $\F(E,u)$, where $E$ is a set of finite perimeter and $u\in L^1(\pared E, \R_+)$ is studied. These energies have been recently considered in the context of materials science to derive a new model in crystal growth that takes into account the effect of atoms freely diffusing on the surface (called \emph{adatoms}), which are responsible for morphological evolution through an attachment and detachment process.
Regular critical points, existence and uniqueness of minimizers are discussed and the relaxation of $\F$ in a general setting under the $L^1$ convergence of sets and the vague convergence of measures is characterized.
This is part of an ongoing project aimed at an analytical study of diffuse interface approximations of the associated evolution equations.
\end{abstract}

  \section{Introduction}
In this paper we investigate the behavior of a surface energy of the form
	\begin{equation}\label{en}
	\F(E,u):=\int_{\pa E} \psi(u) \d \H^{n-1} 
	\end{equation}
and in particular we characterize its lower semi-continuous envelope. Here $\psi:\R_+ \rightarrow (0,\infty)$ is a convex function, $\R_+:=[0,\infty)$, $E \subset \R^n$, a smooth set, represents the region occupied by the \textit{crystal} and $u\in L^1(\pa E,\R_+)$ is a Borel function representing the \emph{adatom density}.

The above quantity, proposed by Burger in \cite{Burger}, is the underlying energy for the evolution equations
	\begin{equation}\label{evo}
	\left\{\begin{array}{ll}
	\pa_t u + (\rho+uH_{\partial E_t})V =D\Delta_{\pa E_t} \psi'(u) & \text{on $\pa E_t$},\\
	bV+\psi H-(\rho+uH_{\partial E_t})\psi'(u)=0 & \text{on $\pa E_t$},
	\end{array}	
	\right.
	\end{equation}
where $\{E_t\}_{t\in I}$ are evolving smooth sets, $V$ is the normal velocity to $\pa E_t$, $H_{\partial E_t}$ is its mean curvature, $u(\cdot,t):\pa E_t\rightarrow \R_+$ is the \textit{adatom density} on $\pa E_t$, $\rho > 0$ is the constant volumetric mass density of the crystal, $b>0$ is a constant called \textit{kinetic coefficient} and $D>0$ is the \emph{diffusion coefficient} of the adatoms.
The above system of evolution equations is a refinement of the classical model for \emph{surface diffusion}, one of the most important mechanisms for crystal growth (see \cite{Taylor}), which, according to the Einstein-Nernst relation,
can be written as
\begin{equation}\label{eq:standardevo}
\rho V - D\Delta_{\pa E_t} \mu = F\cdot \nu  \quad\text{on $\pa E_t$}\,.
\end{equation}
Here $\mu$ denotes the chemical potential and $F$ represents the deposition flux on the surface (in \eqref{evo}, $F\equiv0$).
The evolution equation \eqref{eq:standardevo} and the corresponding energy has been widely used to study properties of crystal growth from an analytic point of view (see  \cite{Bonacini, Bonacini1, BonChamb, Capriani, FFLM, FFLM2, FFLM3, FuscoMorini}).
Nonetheless, it does not take into consideration the effect of the atoms freely diffusing on the surface (called \emph{adatoms}), which are responsible for surface evolution through an attachment and detachment process.
Taking into account their role is a relatively new feature in mathematical models.
System \eqref{evo} was introduced first by Fried and Gurtin \cite{Gurtin} a decade ago.
It accounts also for the kinetic effects through the term $bV$, that represents a dissipative force associated to these attachments and detachments.
To focus on the role of adatoms, \eqref{en} is a surface energy depending only on $u$, neglecting the elastic bulk and anisotropic surface terms that are usually considered in the study of \eqref{eq:standardevo}.
Thus, in our case, the chemical potential $\mu$ reduces to $\psi'(u)$.

So far, the only analytical results about \eqref{en} and \eqref{evo} have been obtained in \cite{Burger}, where a
study of critical points and minimizers is presented and where the dynamics are studied in two dimensions near equilibrium configurations.
In order to perform numerical simulation on the system \eqref{evo}, in \cite{RV} (in the particular case in which $\psi(s)=1+s^2/2$) the authors introduce a diffuse interface approximation based on the energy
\begin{equation}\label{eq:phasefield}
\F_\varepsilon(\phi,u):=\int_{\R^n}\left(\, \frac{\varepsilon}{2}|\nabla \phi|^2 +\frac{1}{\varepsilon}G(\phi)  \,\right)
    \psi(u) \d x\,.
\end{equation}
(here $G$ is a double well potential) and show formal convergence of the associated evolution equations to \eqref{evo}. Numerical analysis based on a level set approach is carried out in \cite{RV2}.\\

Our paper is a first step of an ongoing project in studying analytically the above convergence.
In the spirit of the work by Taylor (\cite{Taylor2}) and Cahn-Taylor \cite{CahnTaylor}, the idea is to see the approximate evolution equations proposed in \cite{RV} as a gradient flow of \eqref{eq:phasefield}
and to obtain information about the limiting equations by using $\Gamma$-convergence techniques (see \cite{DeGFra, DM, Bra}).
A natural question is whether $\F_\varepsilon$ $\Gamma$-converges in some suitable topology to $\F$.
For this reason, we rewrite the energy \eqref{en} within the context of sets of finite perimeter and Radon measures, and set
\[
\F(E,\mu):=\int_{\pared E} \psi(u)\d \H^{n-1}\,,
\]
when the measure $\mu$ is absolutely continuous with respect to $\H^{n-1}\restr\pared E$ and $u$ is the Radon-Nikodym derivative with respect to $\H^{n-1}\restr\pared E$, and $+\infty$ otherwise.
Here $\pared E$ is the reduced boundary of $E$ (see \cite{AFP}, \cite{Maggi}, that coincides with
$\partial E$ in the case of smooth sets). We adopt a natural topology given by the $L^1$ convergence of sets and the weak*-convergence of measures.
We show that in general $\F$ fails to be lower semi-continuous (see Corollary \ref{cor:neccond}) for that topology. To be precise, our main result can be stated as follows (see Theorem \ref{thm:relax}).

\begin{theorem}\label{mainthm}
Let $\psi:\R_+ \rightarrow(0,\infty)$ be a non-decreasing convex function.
The lower semi-continuous envelope of $\F$ is
\[
\ov{\F}(E,\mu):= \int_{\pared E} \ov{\psi}\left( u \right)\d\H^{n-1} + \Theta\mu^s(\R^n)\,,
\]
where $\ov{\psi}$ is the convex subadditive envelope of $\psi$ (see Definition \ref{def:convexsubadditiveenvelope}), and $\Theta:=\lim_{s\rightarrow\infty}\ov{\psi}(s)/s$.
Here $\mu=u \H^{n-1}\restr\pared E+\mu^s$ is the Radon-Nikodym decomposition of $\mu$.
\end{theorem}

The novelty of this result relies on the fact that we allow both $\pared E$ and $\mu$ to vary.
To our knowledge, in the literature, results in this context involve either a better convergence for the measures (see \cite{ButtFredd}), a fixed reference measure (see Bouchitt\'{e}-Buttazzo \cite[Section 3.3]{BouchitteButtazzoBook}, and Fonseca \cite{Irene}) or consider integrands depending on the jump of a $BV$ function and the normal to its jump set (see \cite[Section 5]{AFP}).

In the relaxation $\ov{\F}$ of $\F$, we obtain the \textit{convex subadditive envelope} of $\psi$, since subadditivity and convexity are necessary conditions for lower semi-continuity, issuing from oscillation phenomena (see Corollary \ref{cor:neccond}). In turn, concentration effects lead to the recession part $\Theta\mu^s$.
The key ingredient in our construction of the recovery sequences, where $\psi > \ov{\psi}$ (we recall that $\ov{\psi}\leq\psi$), is an interplay between increasing the perimeter and decreasing the adatom density accordingly. This is done in Proposition \ref{prop:density} and Lemma \ref{lem:wrigg}. As a consequence, we also obtain the following general fact which can be seen as a local estimate of the lack of upper semi-continuity of the perimeter in $L^1$. 

  \begin{theorem}\label{thm: main theorem wriggling}
  Let $E$ be a set of finite perimeter in $\R^N$ and $f\in L^1(\pared E,\R_+)$. Here $L^1(\pared E,\R_+)$ is meant with respect to the $\H^{n-1}\restr\pared E$ measure. Then, there exists a sequence $(E_k)_{k\in \N}$ of bounded, smooth sets of finite perimeter such that
 $\ca_{E_k}\rightarrow \ca_{E}$ in $L^1$ and
   \begin{align*}
  	\lim_{k\rightarrow +\infty} P(E_k;A) = P(E;A)+\int_{\pared E\cap A} f \d \H^{n-1} 
  	\end{align*}
  for all open sets $A$ in $\R^n$ such that $\H^{n-1}(\pa A \cap \pared E) = 0$.
  \end{theorem}
  
\noindent It is worth noticing that with $f\equiv\alpha$ we get $P(E_k;A)\rightarrow (1+\a)P(E;A)$. The non triviality of the above results relies on the fact that the sequence $(E_k)_{k\in \N}$ does \textit{not} depend on $A$. 

We also investigate critical points and minimizers of $\F$ and $\ov{\F}$ under a total mass constraint
$$\rho |E| + \int_{\pared E} u \d \H^{n-1} = m.$$
In Proposition \ref{prop:charregcrit} we define a notion of \textit{regular critical points} of $\F$ and if $\psi$ is strictly convex and of class $\mathcal{C}^1$ we characterize them as the balls with constant adatom density $c$ satisfying 
$$(\psi(c) - c\psi'(c))H_{\pa E} = \rho \psi'(c)$$
where $H_{\pa E}$ denotes the mean curvature of $\pa E$.
The above condition can be written as
\[
H_{\pa E}\psi(c)-\psi'(c)\rho_{\text{eff}}=0\,,
\]
where $\rho_{\text{eff}}:=\rho+cH_{\pa E}$ plays the role of an \emph{effective density}, as can be seen in \eqref{evo}.
In Theorem \ref{thm:ex}, we provide sharp assumptions on $\psi$ to ensure that the constrained minimum of $\F$ can be reached by a ball with constant but non-zero adatom density. Nonetheless in Proposition \ref{prop:plateau} we show that the energy restricted to those couples can exhibit a plateau of minimizers even if $\psi$ is strictly convex.
For what concerns $\ov{\F}$, in Theorems \ref{critfbar} and \ref{rolexx} we define corresponding notions of regular critical points and constrained minimizers and show that the above results still hold for the \textit{absolutely continuous part} $(E,u)$ of $(E,\mu)$ if $|E| > 0$.

It is interesting to notice that due to the structure of the problem we are able to prove existence of minimizers without using the Direct Method of the Calculus of Variations. However, for the sake of completeness a compactness result for sequences of bounded energy is proven in the Appendix (Theorem \ref{app:comp}).

Finally, we would like to point out that the \textit{parabolicity condition}
\begin{equation}\label{paracond}
\psi(s) - s\psi'(s) \geq 0
\end{equation}
plays a central role in our analysis, as it defines $\ov{\psi}$ (see Remark \ref{recipe}) and appears in different other contexts. It was introduced in \cite{Burger} as a stability condition and appears as a parabolicity condition in the evolution equations, as we will discuss in a forthcoming paper about the aforementioned $\Gamma$-convergence-type analysis and associated evolution equations. In particular, by adapting the method developed in the current paper, we will show that $\F_\varepsilon$ $\Gamma$-converges to $\ov{\F}$. \\

The organization of this paper is as follows: in Section 2 we recall some basic facts that we will use throughout the paper. Section 3 deals with critical points and the study of constrained minimizers. Section 4 is the central part of this paper, and is where we prove Theorem \ref{mainthm}. Section 5 studies minimizers of the relaxed functional. Finally, in the appendix we prove some basic facts about the convex subadditive envelope of a function and present some additional and general results derived from Section 4.


  \section{Preliminaries}
    
We collect here the basic notions and notations we will use throughout the paper. 

\subsection{Sets of finite perimeter}

We start by recalling the basic notions of set of finite perimeters, which can be found in \cite[Section 3]{AFP} and \cite[Section 11]{Maggi}.
  
\begin{definition}
Let $E$ be an $\mathcal L^n$ measurable set of $\mathbb R^n$. We call \textit{perimeter} of $E$ in $\mathbb R^n$
$$P(E):= \sup \left\{ \int_E \text{div}(\phi) \mathrm dx \enskip :  \enskip \phi \in  \mathcal C^1_c(\mathbb R^n, \mathbb R^n),\  \| \phi \|_\infty \leq 1 \right\}.$$
We say that $E$ is a \textit{set of finite perimeter} if $|E|<\infty$ and $P(E) < \infty$.

We will denote by $\mathfrak C (\R^n)$ the family of all sets of finite perimeter in $\mathbb R^n$.
\end{definition}

\begin{remark}
If $E$ is a set of finite perimeter, then its characteristic function $\ca_E \in  BV(\mathbb R^n)$ is of \textit{bounded variation}. Its distributional derivative $D \ca_E$ is a $\mathbb R^n$-valued finite Radon measure on $\mathbb R^n$. We will write $|D \ca_E|$ for its total variation measure.
\end{remark}

\begin{definition}
For any Borel set $F\subset \mathbb R^n$ the \textit{relative perimeter of $E$ in $F$} is defined as:
$$P(E;F) = |D\ca_E|(F).$$
\end{definition}

\begin{definition}
Let $E \subset \mathbb R^n$ be a set of finite perimeter. The \textit{reduced boundary of $E$} is the set
$$ \pa^* E := \left\{x\in \text{supp} |D \ca_E|\  :\  \exists\lim_{r\to 0} \frac{D \ca_E (B_r(x))}{|D \ca_E|(B_r(x))}=:\nu_E(x)\,,\quad |\nu_E(x)| = 1. \right\} $$  
\end{definition}

\begin{remark}
It is well known that the reduced boundary of a set of finite perimeter is an $n-1$ rectifiable set and 
$$|D \ca_E| = \mathcal H^{n-1} \llcorner \pa^* E, \quad D \ca_E = |D \ca_E|\nu_E.$$ 
Moreover, the following generalized Gauss-Green formula holds true
\begin{equation*}
\int_E \text{div }T \ \mathrm dx = -\int_{\pa^* E} T \cdot \nu_E \ \mathrm d\mathcal H^{n-1}
\end{equation*}
for every $T \in \mathcal C^1_c(\mathbb R^n, \mathbb R^n)$.
\end{remark}

\subsection{Smooth manifolds}

Here we recall some differentiability and integrability results for smooth manifolds. For a reference, see \cite[Section 2.10]{AFP} and \cite[Section 8]{Maggi}.

\begin{definition}\label{def:tangdiff}
Let $M \subset \mathbb R^n$ be a $\mathcal C^1$ hypersurface and let us denote by $T_x M$ the tangent space to $M$ at $x\in M$.
A function $f: \mathbb R^n \to \mathbb R^m$ is said to be \textit{tangentially differentiable with respect to $M$ at x} if the restriction of $f$ to $x+T_x M$ is differentiable at $x$, and we will call $\nabla^M f(x)$ an associated Jacobian matrix. Moreover, if $f: \mathbb R^n \to \mathbb R^m$ is tangentially differentiable at $x\in M$, we define the \textit{tangential jacobian of $f$ with respect to $M$ at $x$} as
$$ J^M f(x) := \sqrt{\det \left([\nabla^M f(x)]^T \nabla^M f(x) \right)}\,,$$
where $[\nabla^M f(x)]^T$ denotes the transpose matrix of $\nabla^M f(x)$. 
\end{definition}

\begin{theorem}\label{thm:areaformula}
Let $M \subset \mathbb R^n$ be a $\mathcal C^1$ hypersurface and let $f: \R^n \to \R^n$ be an injective $\mathcal C^1$ function.
Then, the following \emph{area formula} holds 
\begin{equation}\label{areaformula}
 \H^{n-1} (f(M))  = \int_M J^M f(x) \d \H^{n-1}(x)\,.
 \end{equation}
Moreover, if $g:\R^n\rightarrow[0,\infty]$ is a Borel function, then also the following change of variable formula holds
\begin{equation}\label{eq:changeofvariable}
\int_{f(M)} g(y)\d\H^{n-1}(y)=\int_M g\left(f(x)\right)J^M f(x) \d\H^{n-1}(x)\,.
\end{equation}
\end{theorem}

\begin{definition}
Let $M \subset \mathbb R^n$ be a $\mathcal C^1$ hypersurface. We say that a vector field $T:M\rightarrow\R^n$ is \emph{tangential} to $M$ if
$T(x)\in T_x M$ for every $x\in M$.
We say that the vector field $T$ is \emph{normal} to $M$ if
$T(x)\perp T_x M$ for every $x\in M$.
\end{definition}

\begin{definition}\label{def:curv}
Given a $\mathcal C^2$ hypersurface  without boundary $M \subset \mathbb R^n$ and a unit normal vector field $\nu_M: M \to \mathbb S^{n-1}$, there exists a normal vector field $\textbf{H}_M\in \mathcal C^0(M, \R^n)$ such that
\begin{equation}\label{def:H_M}
 \int_M \nabla^M \phi \ \mathrm d \mathcal H^{n-1} = \int_M \phi \ \textbf{H}_M \mathrm d \mathcal H^{n-1}
 \end{equation}
for every $\phi \in \mathcal C^1_c (\mathbb R^n)$.
$\textbf{H}_M$ is called the \textit{mean curvature vector field} of $M$. Up to the orientation choice, this defines the scalar mean curvature $H_M$ through $$H_M \nu_M:=\textbf{H}_M.$$
\end{definition}

\begin{definition}\label{def:tangdiv}
Given a $\mathcal C^2$ hypersurface  without boundary $M \subset \mathbb R^n$ and a vector field $T \in \mathcal C^1_c(\mathbb R^n, \mathbb R^n)$ we define the \textit{tangential divergence of $T$ on M} by 
$$ \text{div}^M T := \text{div } T - (\nabla T  \nu_M) \cdot \nu_M = \text{tr}( \nabla^M T).$$
This provides another formulation of \eqref{def:H_M} as
\begin{equation}\label{def:tangdivformula}
 \int_M  \text{div}^M T \ \mathrm d \mathcal H^{n-1} = \int_M T \cdot \textbf{H}_M \mathrm d \mathcal H^{n-1}
 \end{equation}
for all $T\in \mathcal C^1_c (\mathbb R^n, \mathbb R^n)$.
\end{definition}

Choosing $T= \nu_M$ in \eqref{def:tangdivformula} and localizing around any point of $M$ we obtain the well known relation
\begin{equation*}
\text{div}^M (\nu_M) = H_M.
\end{equation*}

We adopt the convention of outward normal derivatives so that balls have positive curvature.
Finally, we recall the product formula for the divergence of \textit{tangential} vector fields.

\begin{proposition}
Under the assumptions of Definition \ref{def:tangdiv}, if
$T\in\mathcal{C}^1_c(\R^n,\R^n)$ is \textit{tangential}, that is $T(x) \in T_x M$ at all points, then for all $\phi \in \mathcal C^1(\mathbb R^n)$
\begin{equation}\label{tangdivprod}
\text{div}^M (\phi T) = \phi \ \text{div}^M T + \nabla^M  \phi\cdot T\,.
\end{equation}
This yields the integration by parts formula
\begin{equation}\label{ibpsurface}
 \int_M  \phi \ \text{div}^M T \ \mathrm d \mathcal H^{n-1} = - \int_M \nabla^M \phi \cdot T \mathrm d \mathcal H^{n-1}\,.
\end{equation}
\end{proposition}

\subsection{Radon measures}
Finally, we recall some basic properties of Radon measures that we will use in Section~\ref{sec:relax}.
For a reference see, for instance, \cite[Section 1.4]{AFP}, \cite[Section 2]{Maggi}.

\begin{definition}
\label{def:weakstar}
We denote by $\mathcal{M}^+_{loc}(\R^n)$ the space of locally finite non-negative Radon measures.
we say that a sequence $(\mu_k)_{k\in  \mathbb N} \subset \mathcal{M}^+_{loc}(\mathbb R^n)$ is \emph{locally weakly*-converging} to $\mu\in \mathcal{M}^+_{loc}(\R^n)$ if
\[
\lim_{k\to +\infty} \int_{\mathbb R^n} \phi \mathrm \  d\mu_k = \int_{\mathbb R^n} \phi \mathrm \  d\mu
\]
for every $\phi  \in \mathcal C_c(\mathbb R^n)$.
In this case, we will write $\mu_k \wt \mu$. 
\end{definition}

A useful continuity property for sequences of locally weakly*-convergent measures is the following.

\begin{lemma}
Let $(\mu_k)_{k\in  \mathbb N} \subset \mathcal{M}^+_{loc}(\mathbb R^n)$, $\mu\in\mathcal{M}^+_{loc}(\R^n)$ be such that $\mu_k \wt \mu$. Then
\begin{equation}\label{borelconv}
\lim_{k\to +\infty} \mu_k(E) = \mu(E)\,,
\end{equation}
for all bounded Borel sets $E \subset \mathbb R^n$ for which $\mu(\pa E) = 0$.
In particular, for any $x\in \mathbb R^n$ it holds that
\begin{equation}\label{foliation}
\lim_{k\to +\infty}\mu_k(B_r(x))=\mu( B_r(x))\,,
\end{equation}
for all but countably many $r>0$.
\end{lemma}

The following compactness result for finite Radon measures holds.

\begin{lemma}\label{weak*comp}
Let $(\mu_k)_{k\in  \mathbb N} \subset \mathcal{M}^+_{loc}(\mathbb R^n)$ be such that
\begin{equation*}
\sup_{k\in  \mathbb N} \mu_k(\R^n) < \infty\,.
\end{equation*}
Then there exists a subsequence of $(\mu_k)_{k\in  \mathbb N}$
that locally weakly*-converges to some $\mu \in \mathcal{M}^+_{loc}(\mathbb R^n)$.
\end{lemma}

Finally, we recall that the space $\mathcal{M}^+_{loc}(\R^n)$ is a (separable) metric space
(for a proof, see, for instance, \cite[Proposition 2.6]{DeLellis}).

\begin{proposition}\label{prop:weakstarmetriz}
The weak*-convergence on $\mathcal{M}^+_{loc}(\R^n)$ is metrizable by a distance that we will denote $d_{\mathcal M}$. In particular, it holds that
\[
\mu_k\wt\mu\quad\Leftrightarrow\quad
\lim_{k\rightarrow\infty}d_{\mathcal M}(\mu_k,\mu)=0\
\]
\end{proposition}


\section{The constrained minimization problem}

Throughout the paper we will assume the following.

  \subsection{Setting}
  
\begin{definition}\label{def:admipsi}
Let $\psi:\R_+\rightarrow(0,+\infty)$, be convex and $\mathcal{C}^1$ with
\[
0<\psi(0)<\psi(s)
\]
for every $s>0$. 
\label{def:energy}
We define the energy functional
\[
\F(E,u):=\int_{\pared E} \psi(u) \d \H^{n-1}\,,
\]
where $E\subset\R^n$ is a set of finite perimeter and $u\in L^1(\pared E,\R_+)$ is a Borel function. Here the space $L^1(\pared E,\R_+)$ is meant with respect to the $\H^{n-1}\restr\pared E$ measure.
\end{definition}    

We are interested in studying the optimal shapes and adatom distributions (the function $u$) under a total mass constraint.

\begin{definition}
For $m>0$, define
\begin{equation}\label{eq:pb}
\gamma_m:=\inf\{\, \F(E,u) \,:\, (E,u)\in \mathrm{Cl}(m) \,\}\,,
\end{equation}
where
\[
\mathrm{Cl}(m):=\left\{\, (E,u) \,:\,  E \text{ set of finite perimeter}, \ u\in L^1(\pared E,\R_+)\,, \
    \mathcal J(E,u)=m \,\right\}\,,
\]
and
\begin{equation*}
\mathcal J(E,u):=\rho |E| + \int_{\pared E} u \d\H^{n-1}\,.
\end{equation*}
Here $\rho > 0$ is a constant that denotes the volumetric mass density of the crystal.
\end{definition}


\subsection{Critical points}

We start our investigation by studying the properties of critical points of the energy.
To this aim we need to perform variations of a given couple $(E,u)\in\mathrm{Cl}(m)$ that satisfies the constraint.
We show in Appendix \ref{cstmasscurve} that it is enough to consider variations that preserve the constraint only at the first order (see Remark \ref{rem:just}).

\begin{definition}\label{def:av}
Let $(E,u)\in\mathrm{Cl}(m)$ with $E$ a bounded set of class $\mathcal{C}^3$ and $u(x)\geq\tau$ for $\H^{n-1}$-a.e. $x\in\partial E$, for some $\tau>0$. We define the set of \textit{admissible velocities} for $(E,u)$ as
\[
\mathrm{Ad}(E,u):=\left\{ (v,w)\in \mathcal{C}^1_b(\pa E)\times \mathcal{C}^1_b(\pa E) :
	\int_{\pa E} [w+v(uH_{\pa E}+\rho)]\d\H^{n-1}=0\right\}
\]
where $\mathcal C^1_b$ means $\mathcal C^1$ and bounded functions and $H_{\pa E}$ is given in Definition \ref{def:curv}.
\end{definition}

By using the above admissible velocities, it is possible to derive the Euler-Lagrange equations for $\F$. For that, we will need to apply Lebesgue's dominated convergence theorem and thus make use of the following technical growth assumption.

\begin{itemize}
\item [(H)] There exists $p\geq 1$ and $A, B > 0$ such that $$\psi(s),\, \psi'(s) \leq A + Bs^p \,\, \text{ for all } s\geq 0$$ and $u \in L^p(\pa E,\R_+)$.
 \label{growthass}
\end{itemize}

\begin{proposition}\label{prop:EL}
Let $(E,u)\in\mathrm{Cl}(m)$ be as in the previous definition and let $(v,w)\in \mathrm{Ad}(E,u)$.
Assume moreover that $(H)$ holds.
Then, the first variation of the functional $\F$ computed at $(E,u)$ with respect to the variations \eqref{eq:vare} and \eqref{eq:varu} is given by
\begin{equation*}
\frac{d}{d t}\Big{|}_{t=0} \F(E_t,u_t) = \int_{\pa E} [\,\psi'(u) w+\psi(u) v H_{\pa E}\,] \d \H^{n-1}\,.
\end{equation*}
\end{proposition}

The main result of this section is a characterization of the regular critical points. This extends a result proved in \cite{Burger} by using the evolution equation. Here, we use the Euler-Lagrange equations.

\begin{proposition}\label{prop:charregcrit}
Let $(E,u)\in\mathrm{Cl}(m)$ be a regular critical point for $\F$, \textit{i.e.}, $(E,u)$ is as in Definition \ref{def:av} and satisfies 
\begin{equation}\label{eq:crit}
\int_{\pa E} [\,\psi'(u) w+\psi(u) v H_{\pa E} \,] \d \H^{n-1}=0 \quad\, \text{ for all } (v,w)\in \mathrm{Ad}(E,u).
\end{equation}
Assume that $\psi$ is strictly convex. Then $E$ is a finite disjoint union of balls $\bigcup_{i=1}^m B_i$ with same curvature $H_{\pa B}$ and $u$ is a constant $c$ such that
\begin{equation}\label{cstcritpointsvalue}
(\psi(c)-c\psi'(c))H_{\pa B} = \rho \psi'(c)\,.
\end{equation}
Conversely, any such $(\bigcup_{i=1}^m B_i,c)$ is a regular critical point.

Finally, if $(\bigcup_{i=1}^m B_i,c)$ is a regular critical point, then (see Remark \ref{recipe})
\begin{equation*}
0 < c < s_0 :=\sup\{s\in \R_+ \ : \ \psi(s)-\psi'(s) s  > 0\}\,.
\end{equation*}
\end{proposition}

\begin{remark}\label{rem:just}
In order to justify our definition of admissible variations, we argue as follows: take a bounded set $E$ of class $\mathcal{C}^2$
and denote by $\nu_E$ the exterior normal to $E$ on $\partial E$.
Let us denote by $\mathrm{d}(y,\pa E)$ the \emph{distance} of a point $y\in\R^n$ from $\partial E$.
It is well known (see \cite[Section 14.6]{GT}) that it is possible to find $\delta>0$ such that for every point $z$ in the set
\[
(\pa E)_\delta:=\{\, y\in\R^n \,:\, \mathrm{d}(y, \partial E)<\delta \,\}
\]
there exists a unique $\Pi(z)\in\pa E$ such that $\mathrm{d}(z,\pa E)=|z-\Pi(z)|$.
In particular, the \emph{projection} map $\Pi:(\pa E)_\delta\rightarrow\pa E$ is of class $C^1$ and it is possible to write
any $z\in(\pa E)_\delta$ as
\begin{equation}\label{eq:z}
z=\Pi(z)+\mathrm{d}(z,\pa E)\nu_E(x)\,.
\end{equation}
Then, consider the extension of the exterior normal to $(\partial E)_\delta$
given by (with an abuse of notation we make use of the same symbol)
\[
\nu_E(z):=\nu_E(x)\,,
\]
where $z\in (\partial E)_\delta$ is written as in \eqref{eq:z}. The above extension is unique and well defined.

Fix a function $\varphi:(-\delta,\delta)\rightarrow\R$ with $0\leq\varphi\leq1$, such that $\varphi\equiv1$ on
$[-\frac{\delta}{4},\frac{\delta}{4}]$ and $\varphi\in \mathcal{C}^\infty_c([-\frac{\delta}{2},\frac{\delta}{2}])$.
Let $v\in \mathcal{C}^1(\partial E)$  and, for
\begin{equation}\label{eq:t}
|t|<\bar{t}:=
\left\{
\begin{array}{ll}
\frac{\delta/2}{\sup_{\partial E}|v|} & \text{ if } v\not\equiv0\,,\\
&\\
+\infty & \text{ otherwise }\,,
\end{array}\right.
\end{equation}
consider the $\mathcal C^1$ diffeomorphism $\Phi_t:\R^n\rightarrow\R^n$ given by
\begin{equation*}
\Phi_t(z):=
\left\{
\begin{array}{ll}
z+t\varphi\left(\mathrm{d}(z,\pa E)\right)v(x)\nu_E(x) & \text{ if } z\in(\partial E)_\delta \text{ as in } \eqref{eq:z}\,,\\
z & \text{ otherwise in } \R^n\,.
\end{array}
\right.
\end{equation*}
Define, for $|t|<\bar{t}$, the variations
\begin{equation}\label{eq:vare}
E_t:=\Phi_t(E)\,.
\end{equation}
Now let $w\in \mathcal{C}^1(\partial E)$ and set $u_t:\partial E_t\rightarrow\R$ as
\begin{equation}\label{eq:varu}
u_t(y):=u(\Phi_t^{-1}(y))+t w(\Phi_t^{-1}(y))\,.
\end{equation}
We want the mass constraint to be satisfied at the first order, \emph{i.e.},
\[
\frac{d}{d t}\Big{|}_{t=0} \mathcal J(E_t,u_t)=0\,.
\]
Moreover, to preserve positivity of $u_t$ without further restricting the admissible velocities, we require $u\geq\tau>0$ on $\partial E$. 
It is well known that (see \cite{Maggi}, Proposition 17.8)
\[
\frac{d}{d t}\Big{|}_{t=0} |E_t|=\frac{d}{d t}\Big{|}_{t=0} |\Phi_t(E)|=\int_{\partial E} \nu_E\cdot \frac{\pa \Phi_t}{\pa t}\Big|_{t=0}  \d \H^{n-1} =  \int_{\partial E} v \d \H^{n-1}.
\]
By the change of variable formula (see \eqref{eq:changeofvariable}) we can write
\begin{align*}
\int_{\pa E_t} u_t(y)\d \H^{n-1}(y)&=\int_{\pa E} u_t(\Phi_t(x))J^{\pa E}\Phi_t(x) \d \H^{n-1}(x)\\
&=\int_{\pa E} [\,u(x)+t w(x) \,]J^{\pa E}\Phi_t(x) \d \H^{n-1}(x)\,,
\end{align*}
where $J^{\pa E}\Phi_t$ is given in Definition \ref{def:tangdiff}.
Using the fact that (see \cite[(17.30)]{Maggi})
\begin{equation}\label{botta}
\frac{d}{dt}\Big{|}_{t=0} J^{\pa E} \Phi_t=\dive^{\partial E}(v \nu_E)=vH_{\pa E}\,,
\end{equation}
we obtain
\[
\frac{d}{d t}\Big{|}_{t=0} \int_{\pa E_t} u_t(y)\d \H^{n-1} = \int_{\pa E} \left[\,w  + u vH_{\pa E}\,\right] \d \H^{n-1}\,,
\]
and thus
\begin{equation*}
\frac{d}{d t}\Big{|}_{t=0} \mathcal J(E_t,u_t)=\int_{\pa E} \left[\,w  + v(u H_{\pa E} + \rho) \,\right] \d \H^{n-1}\,.
\end{equation*}
This justifies our definition of the set of admissible velocities $\mathrm{Ad}(E,u)$:
it can be seen as (part of) the \emph{tangent space} to $\mathrm{Cl}(m)$ at the point $(E,u)$.
\end{remark}

\begin{proof}[Proof of Proposition \ref{prop:EL}]
We have
\begin{align*}
\frac{d}{d t}\Big{|}_{t=0} \F(E_t,u_t)	&=\frac{d}{d t}\Big{|}_{t=0} 
	\int_{\pa E_t} \psi(u_t(y)) \d \H^{n-1}(y)\\
&=\frac{d}{d t}\Big{|}_{t=0} \int_{\pa E} \psi(u_t(\phi_t(x))) J^{\pa 
	E} \phi_t(x) \d \H^{n-1}(x)\\
&=\frac{d}{d t}\Big{|}_{t=0} \int_{\pa E} \psi\left(u(x)+tw(x)\right) J^{\pa 	E} \phi_t(x) \d \H^{n-1}(x)\\
&=\int_{\pa E} \left[\,\psi'\left(u(x)\right) w(x)+\psi\left(u(x)\right) v(x) H_{\pa E}(x) \,\right] \d \H^{n-1}(x)\,,
\end{align*}
where in the last equality we have used \eqref{botta} and Lebesgue's dominated convergence theorem thanks to (H) and the fact that $v, w$ and $H_{\pa E}$ are bounded.
\end{proof}

\begin{proof}[Proof of Proposition \ref{prop:charregcrit}]
We divide the proof in four steps.\\

\textbf{Step one}: \textit{$u$ is constant on each connected component of $\pa E$}. Take a tangential vector field $T\in\mathcal{C}^2_c(\pa E,\R^n)$. Then by \eqref{ibpsurface}, $(v,w):=(0,\dive^{\pa E}(T))\in \mathrm{Ad}(E,u)$. Since $(E,u)$ satisfies \eqref{eq:crit}, using \eqref{ibpsurface} we get
\[
0=\int_{\pa E} \psi'(u) \dive^{\pa E}(T)\d \H^{n-1}\,.
\]
Using a density argument, we see that the above equality holds also for every $T\in\mathcal{C}^1_c(\pa E,\R^n)$.
Using the fact that $T$ is an arbitrary tangential vector field, we conclude that $\nabla^{\pa E} \left(\psi'(u)\right)=0$ on $\pa E$ in the sense of distributions, which implies that $\psi'(u)$ is constant on each connected component of $\pa E$. By the strict convexity of $\psi$, $u$ is constant on each connected component of $\partial E$.\\

\textbf{Step two}: \textit{$H_{\pa E}$ is constant on each connected component of $\pa E$, which are spheres}.
Let $\pa E_i$ be a connected component of $E$. Let $v\in \mathcal{C}^1(\pa E_i)$ and consider the admissible velocities defined as $(v,-v(uH_{\pa E_i}+\rho))$ on $\pa E_i$ and $(0,0)$ on other connected components. Using the fact that $u$ is a constant $c_i$ on $\pa E_i$, by \eqref{eq:crit} we obtain
\begin{align}\label{eq:v}
0 &=-\psi'(c_i)\int_{\pa E_i} v(c_iH_{\pa E_i}+\rho)\d \H^{n-1}
    +\psi(c_i)\int_{\pa E_i} v H_{\pa E_i}  
	\d \H^{n-1}\nonumber \\
&=\left( \psi(c_i)-c_i\psi'(c_i) \right) \int_{\pa E_i} vH_{\pa E_i}\d\H^{n-1}
    - \rho\psi'(c_i)\int_{\pa E_i}v\d\H^{n-1}\,. 
\end{align}
We claim that $\psi(c_i)-c_i\psi'(c_i)\neq0$. Indeed, assume it is zero.
Then, using \eqref{eq:v} with a non-zero average $v$ we have $\psi'(c_i)=0$ and thus $\psi(c_i)=0$, which is impossible, since
$\psi(s)>0$ for all $s\geq0$.
In order to conclude, take $v\in \mathcal{C}^1(\pa E_i)$ with zero average.
Using again \eqref{eq:v}, we get
\[
\left( \psi(c_i)-c_i\psi'(c_i) \right) \int_{\pa E_i} vH_{\pa E_i}\d\H^{n-1}=0\,,
\]
and so
\[
\int_{\pa E_i} vH_{\pa E_i}\d\H^{n-1}=0\,.
\]
Since this is valid for all $v\in \mathcal{C}^1(\pa E_i)$ with zero average, we conclude that $H_{\pa E_i}$ is a constant.
Finally, the fact that we are assuming $\ov{E}$ to be compact allows us to conclude that each connected component of $\pa E$ is a sphere by using Alexandrov's theorem \cite{Alexandrov}. \\

\textbf{Step three}: \textit{connectedness and bounds on $u$}.
Assume that $\pa E$ has at least two connected components that we denote $(\pa E)_1$ and $(\pa E)_2$.
Let $c_1, c_2$ be the values of the adatom density in $(\pa E)_1$ and $(\pa E)_2$ respectively. Moreover, we will denote by $H_1, H_2$ the constant curvature of  $(\pa E)_1$ and $(\pa E)_2$ respectively.
Consider admissible velocities $(v,w)$ that are equal to $(v_1,w_1)$ on  $(\pa E)_1$, $(v_2,w_2)$ on  $(\pa E)_2$ and identically zero on all other connected components. Using the Definition \ref{def:av} and computations similar to the ones of the previous steps, we get
\begin{align}\label{jeanlassalle}
\nonumber\int_{(\pa E)_1} &w_1 \d \H^{n-1} + (c_1 H_1+\rho)   \int_{(\pa E)_1} v_1 \d \H^{n-1}  \\
&\hspace{1cm}=- \int_{(\pa E)_2} w_2 \d \H^{n-1} - (c_2 H_2+\rho)  \int_{(\pa E)_2} v_2 \d \H^{n-1}\,.
\end{align}
Similarly, as $(E,u)$ is critical, using Step 1 and Step 2 above, the criticality condition \eqref{eq:crit} can be written as
\begin{align}\label{philippepoutou}
\nonumber\psi'(c_1) \int_{(\pa E)_1} & w_1 \d \H^{n-1} + H_1 \psi(c_1) \int_{(\pa E)_1}  v_1  \d \H^{n-1} \\ + &\psi'(c_2) \int_{(\pa E)_2}  w_2 \d \H^{n-1} + H_2 \psi(c_2) \int_{(\pa E)_2}  v_2 \d \H^{n-1} = 0
\end{align}
Using \eqref{jeanlassalle} in \eqref{philippepoutou}, we get
\begin{align}\label{eq:3eq}
&\left(\int_{(\pa E)_2}  w_2 \d \H^{n-1} \right) \left[ \psi'(c_2) - \psi'(c_1) \right]\nonumber\\
 &\hspace{0.6cm}+\left(\int_{(\pa E)_2}  v_2 \d \H^{n-1} \right) \left[ H_2 \psi(c_2) - \psi'(c_1)(c_2 H_2 +\rho) \right] \nonumber\\
 &\hspace{1.2cm} - \left(\int_{(\pa E)_1}  v_1 \d \H^{n-1} \right)
\left[ \psi'(c_1)(c_1 H_1 + \rho) - H_1 \psi(c_1) \right]  =0\,.
\end{align}
Taking $v_1=v_2\equiv0$ and $w_2$ such that
\[
\int_{(\pa E)_2}  w_2 \d \H^{n-1}\neq0\,,
\]
in \eqref{eq:3eq} gives us $\psi'(c_2) = \psi'(c_1)$.
By strict convexity this implies $c_1 = c_2=:c$.
Now, taking $w_2=v_1\equiv0$ and $v_2$ such that
\[
\int_{(\pa E)_2}  v_2 \d \H^{n-1}\neq0\,,
\]
in \eqref{eq:3eq} gives us
$$ (\psi(c) - c \psi'(c)) H_2 = \rho \psi'(c)$$
Exchanging the roles of $v_1$ and $v_2$ provides the same relationship with $H_1$ instead of $H_2$, thus since $\psi(c)-c\psi'(c) \neq 0$ we obtain that $H_1 = H_2 =: H_{\pa B}$ satisfies \eqref{cstcritpointsvalue}. Finally, notice that $H_{\pa B} > 0$, otherwise $E$ would be the complement of a at most countable union of balls, which is impossible since $\ov E$ is compact. In the end, $E = \bigcup_{i=1}^m B_i$ is a disjoint union of balls with the same radius, which is finite since $E$ is compact.
Finally, since $H_{\pa B}, \psi'(c) > 0$ we have $\psi(c) - c \psi'(c) > 0$, which yields $c < s_0$ by definition of $s_0$. \\

\textbf{Step four}: \textit{sufficient conditions}.
Conversely, let $(\bigcup_{i=1}^m B_i,c)$ be a finite disjoint union of balls with constant adatom density and with the same radius satisfying \eqref{cstcritpointsvalue}.
Using Definition \ref{def:av} we get on each connected component $\pa B_i$
\begin{align*}
&\int_{\pa B_i}[\,\psi'(c) w+\psi(c) v H_{\pa B} \,] \d \H^{n-1} \\
&\hspace{2cm}= \left( \int_{\pa B_i} v \d\H^{n-1} \right) \left[ \left(\psi(c) - c\psi'(c)\right) H_{\pa B} - \rho \psi'(c) \right] = 0\,.
\end{align*}
\end{proof}


\subsection{Existence and uniqueness of minimizers}

In this section we address the question of existence and uniqueness of minimizers for the constrained minimization problem
\eqref{eq:pb}. In particular, we prove that the minimum can be achieved by a ball with constant adatom density. A similar result can be found in \cite{Burger}. We present here an alternative proof under more general assumptions and that takes into account also the mass constraint. 

\begin{theorem}\label{thm:ex}
Fix $m>0$. Assume that
\begin{itemize}
\item[(A1)] $\psi'(0)<(n-1)\left(\, \frac{\omega_n}{m} \,\right)^{\frac{1}{n}} \rho^{\frac{1-n}{n}} \psi(0)\,,$
\end{itemize}
where $\omega_n = |B_1|$, and that either one of the following two conditions holds true:
\begin{itemize}
\item[(A2a)] $\psi$ is superlinear at infinity, \emph{i.e.}, $\lim_{s\rightarrow\infty}\psi(s)/s=\infty$,
\item[(A2b)] $\psi(s) =  as + b + g(s)$ with $b\leq0$ and $$ \lim_{s\rightarrow\infty} s^{1/(n-1)}g(s) = \lim_{s\rightarrow\infty} s^{n/(n-1)}g'(s) = 0.$$
\end{itemize}
Then there exist $R\in(0,\ov{R}_m)$, where
\begin{equation*}
\overline{R}_m:= \left(\frac{m}{\rho \omega_n}\right)^{1/n},
\end{equation*}
and a constant $c>0$ such that $(B_R,c)\in\text{Cl}(m)$ and
\[
\F(B_R,c)=\gamma_m\,.
\]
Moreover, if $(E,u)\in\text{Cl}(m)$ is a minimizing couple, then $E$ is a ball, and if $\psi$ is strictly convex, then $u$ is constant.
\end{theorem}

\begin{remark}
Examples of functions satisfying (A1-2) are $\psi(s):= 1+\gamma s^2$ for some $\gamma > 0$ and, less trivially, $\psi(s):= \sqrt{1+s^2}$ when $n\geq 3$. We will later make use of (A2b) for functions that are linear on some interval $(s_0, +\infty)$.
\end{remark}

\begin{remark}
The above theorem does not ensure uniqueness of minimizers, which is false in general (see Proposition \ref{prop:plateau}).
Moreover, in the case hypothesis (A1) or both (A2a) and (A2b) are not satisfied, we will show in Remark \ref{degballs} that the following phenomena can occur:
\begin{itemize}
\item[(i)] there is no minimizer,
\item[(ii)] the minimizer has zero adatom density.
\end{itemize}
Finally we point out that when $\psi$ is not strictly convex there can be a minimizer with non-constant $u$.
\end{remark}

In the sequel we will often use the following reduction lemma.
\begin{lemma}\label{reductionto}
Let $m > 0$. For any $(E,u) \in \text{Cl}(m)$ we have
\begin{equation}\label{VictorHugo}
\F(E,u) \geq \F (B_R, \ov{u} )
\end{equation}
where $$\bar u := \frac{1}{P(E)}\int_{\pared E} u\d \H^{n-1}\,$$ and $B_R$ is a ball such that $\rho |B_R| + \ov{u} P(B_R) = m$. Moreover, \eqref{VictorHugo} is strict unless $E = B_R$. Finally, if $\psi$ is \emph{strictly} convex, then equality is reached if and only if $(E,u) = (B_R, \ov{u})$.
\end{lemma}

\begin{proof}
By Jensen's inequality
\begin{equation*}
\F(E,u)=\int_{\pared E} \psi(u)\d \H^{n-1}
\geq P(E)\psi\left(\bar u\right)=\int_{\pared E}\psi\left(\bar u\right)=\F(E,\overline{u})\,.
\end{equation*}
Notice that if $\psi$ is strictly convex, then equality is reached if and only if $u \equiv \ov{u}$. We can thus replace $u$ by $\bar{u}$ without increasing the energy. Now assume that $E$ is not a ball. Then
by the monotonicity of $r \mapsto \rho |B_r|+\bar{u}P(B_r)$, it is possible to find a radius $R \in(0,\infty)$ such that
\[
\rho |B_{R}|+\bar{u}P(B_{R})=m\,,
\]
that is, $(B_{R},\bar{u})\in\text{Cl}(m)$. We claim that $P(B_{R})<P(E)$, \emph{i.e.}, $\F(B_{R},\bar{u})<\F(E,u)$.  Suppose not. Then by the isoperimetric inequality we would have that
\[
P(B_{R})\geq P(E)>P(B)\,,
\]
where $B$ is a ball with $|B|=|E|$. This implies that $|B_{R}|>|B|=|E|$ and, in turn, that
\[
m=\rho |B_{R}|+\bar{u}P(B_{R})>\rho |E|+\bar{u}P(E)=m\,,
\]
and we reached a contradiction. 
\end{proof}

We now turn to the proof of Theorem \ref{thm:ex}.

\begin{proof}[Proof of Theorem \ref{thm:ex}]
By Lemma \ref{reductionto} we can reduce our study of minimizers to balls with constant adatom density satisfying the constraint. This is a one parameter family. Indeed, for every $R\in(0,\ov{R}_m)$ set
\begin{equation}\label{eq:baru}
\bar u(R):= \frac{m - \rho \omega_n R^n}{n\omega_n R^{n-1}}\,.
\end{equation}
Then $(B_R,\bar{u}(R))\in\text{Cl}(m)$ for every $R\in(0,\ov{R}_m)$. Let
\begin{equation}\label{eq:e}
e(R):=\F(B_R,\overline{u}(R))=n\omega_n R^{n-1} \psi\left(\bar u(R) \right)\,.
\end{equation}
We have
\[
e'(R)=n\omega_nR^{n-2}\left[\, (n-1)\psi\left(\bar{u}(R)\right)-R\psi'\left(\bar{u}(R)\right)
    \left(\, \frac{\rho}{n}+\frac{(n-1)m}{n\omega_nR^n}  \,\right)  \,\right]\,,
\]
and using (A1) we obtain
\[
e'(\ov{R}_m)=n\omega_n\ov{R}_m^{n-2}\left[\, (n-1)\psi(0)-\left(\frac{m}{\rho \omega_n}\right)^{\frac{1}{n}}\psi'(0)\rho  \,\right]>0\,.
\]
Moreover, if (A2a) is satisfied then
\[
e(R)=(m-\omega_nR^n)\frac{\psi\left(\bar{u}(R)\right)}{\bar{u}(R)} \underset{R\to 0}{\longrightarrow} \infty\,,
\]
while if (A2b) holds true, we get
\begin{equation*}
e'(R) = n\omega_nR^{n-2}\left[\,  (n-1)b-\frac{aR\rho}{n} + \underset{R\to 0}{o}(R) \,\right]\,,
\end{equation*}
and thus $e'(R)<0$ for all $R\in(0,\underline{R}_m)$, for some $\underline{R}_m>0$.
This concludes that there exists $R\in(\underline{R}_m,\overline{R}_m)$ such that $e(R)=\F(B_R, \ov{u}(R)) = \g$. 
\end{proof}

\begin{remark}
Notice that the criticality condition $e'(R) = 0$ is equivalent to the general condition \eqref{cstcritpointsvalue} introduced previously.
\end{remark}

\begin{remark}\label{degballs}
Let us consider the function $\psi(s):=as+b$, for some $a,b\in\R$. It holds that
\[
e'(R)=n\omega_n R^{n-2}[b(n-1)-a\rho R]\,.
\]
Taking $b=0$ we get $e'(R)<0$ for all $R\in(0,R_m)$, and thus the minimizer is given by $(B_{R_m},0)$.
If instead we take $b$ with $b>a\rho R/(n-1)$, we get $e'(R)>0$ for all $R\in(0,R_m)$. So, the \emph{expected} minimizer is given by a Dirac delta with infinite adatom density. This is clearly not an admissible minimizer in the present setting (see Section \ref{sec:relax}).
\end{remark}

We now turn to the study of uniqueness of such minimizers. As the next proposition shows, even when $\psi$ is strictly convex there may be a continuum of minimizing balls. 

\begin{proposition}\label{prop:plateau}
For every $0\leq R_1<R_2\leq \ov{R}_m$, there exists a strictly convex function $\psi$ satisfying the assumptions of Definition \ref{def:energy} and such that
\[
\{ R\in(0,R_m) \,:\, e(R)=\gamma_m \}=[R_1,R_2]\,.
\]
\end{proposition}

\begin{proof}
Let $h(R):=-\frac{n-1}{R}$ for $R\in(0,\ov{R}_m]$ and let $f:(0,\ov{R}_m]\rightarrow\R$ be a $\mathcal{C}^1$ negative function with
\[
f(R)=h(R)-\varphi(R)\quad\text{ in } (0,R_1)\,,
\]
\[
f(R)=h(R)\quad\text{ in } [R_1,R_2]\,,
\]
\[
f(R)>h(R)\quad\text{ in } (R_2,\ov{R}_m]\,,
\]
where $\varphi>0$ is such that $\varphi(R)/R^{n-1}\rightarrow0$ as $R\rightarrow0$.
Moreover we will impose that $\|f-h\|_{\C^1}<\e$ for some $\e>0$ that will be chosen later.
Let $g:(0,R_m]\rightarrow\R$ be the solution of the problem
\begin{equation}\label{lessie}
\left\{
\begin{array}{l}
g'(R)=f(R)g(R)\,, \\
g(R_m)=g_m\,,\\
\end{array}
\right.
\end{equation}
for some $g_m>0$. Notice that $g$ is decreasing. We recall that $\bar{u}:(0,R_m]\rightarrow\R_+$ (defined in \eqref{eq:baru}) is invertible, since
\[
\bar{u}'(R)=-\frac{\rho}{n}-\frac{(n-1)m}{n\omega_nR^n}<0\,.
\]
Moreover, $\bar{u}(R_m)=0$ and $\lim_{R\rightarrow0}\bar{u}(R)=\infty$.
Thus, the function $\psi(s):=g\left( \bar{u}^{-1}(s) \right)$ is well defined.
By considering $e:(0,R_m]\rightarrow\R_+$, defined in \eqref{eq:e}, we have that
\[
e'(R)=n\omega_nR^{n-2}\left(\, (n-1)g(R)+Rg'(R)  \,\right)\,,
\]
and thus, by the definition of $g$, it holds that
\[
e'(R)<0\quad\text{ for } R\in(0,R_1)\,,
\]
\[
e'(R)=0\quad\text{ for } R\in[R_1,R_2]\,,
\]
\[
e'(R)>0\quad\text{ for } R\in(R_2,R_m]\,.
\]
Hence
\[
\{ R\in(0,R_m) \,:\, e(E)=\gamma_m \}=[R_1,R_2]\,.
\]
We claim that $\psi$ is strictly convex and satisfies $\psi(s)>\psi(0)>0$. The latter can be seen from the fact that
\begin{equation}\label{eq:gpsi}
g'(R)=\psi'\left(\bar{u}(R)\right)\bar{u}'(R)< 0, \ \ \ \ \bar{u}'(R)<0, \ \  \ \ \psi(0)=g(R_m) = g_m >0
\end{equation}
for all $R\in(0,R_m)$.  
In what concerns strict convexity, by differentiating \eqref{lessie} we get that
\[
\left[\, f^2(R)+f'(R) \,\right]g(R)=g''(R)=\psi''\left(\bar{u}(R)\right)\left(\bar{u}'(R)\right)^2
    +\psi'\left(\bar{u}(R)\right)\bar{u}''(R)\,
\]
and thus, using \eqref{eq:gpsi}, we are led to
\[
\psi''\left(\bar{u}(R)\right)\left(\bar{u}'(R)\right)^2=\left(\, f^2(R)+f'(R)-f(R)\frac{\bar{u}''(R)}{\bar{u}'(R)} \,\right)g(R).
\]
Notice that
\[
h^2(R)+h'(R)-h(R)\frac{\bar{u}''(R)}{\bar{u}'(R)}=\frac{n(n-1)\rho \omega_nR^{n-2}}{\rho \omega_nR^n+(n-1)m}>0\, , 
\]
for all $R\in(0,R_m]$. Then
\begin{align*}
&f^2(R)+f'(R)-f(R)\frac{\bar{u}''(R)}{\bar{u}'(R)} \\
&\hspace{1cm}=(f(R)-h(R))^2+(f(R)-h(R))'\\
&\hspace{2cm}-(f(R)-h(R))\frac{\bar{u}''(R)}{\bar{u}'(R)}+2h(R)(f(R)-h(R))\\
&\hspace{3cm}+h^2(R)+h'(R)-h(R)\frac{\bar{u}''(R)}{\bar{u}'(R)}\,.
\end{align*}
Using $\varphi(R)/R^{n-1}\rightarrow0$ as $R\rightarrow0$ and that
$\bar{u}''(R)/\bar{u}'(R)$ is of order $1/R$ as $R\rightarrow0$,
choosing $\e>0$ small enough we guarantee that $\psi''(s)>0$ for all
$s\in(0,\infty)$.
\end{proof}

\begin{example}
{\rm
If $\psi(s) := 1+\gamma s^2$ for some $\gamma > 0$ and in dimension $n=2$ one can show that $R \mapsto e(R)$ has exactly one critical point $R_*(\gamma)$ which corresponds to the global minimizer, with
$$R_*(\gamma) = \frac{1}{\rho} \sqrt{\frac{2 \sqrt{\gamma^2 m^2 \rho^2-\pi  \gamma m \rho+\pi ^2}+\gamma m \rho-2 \pi }{3 \pi  \gamma}}\,.$$
Notice that $R_*(\gamma)  \underset{\gamma \to +\infty}{\longrightarrow} \ov{R}_m$. A similar asymptotic behavior has been observed also in \cite{Burger} with a misprint in the value of $R_*$ that, however, does not affect the limiting analysis done by the author.}
\end{example}

\section{The relaxed functional}

  \label{sec:relax}
The family $(E,u)$ of couples where $E$ is a set of finite perimeter and  $u \in L^1(\pared E, \R^+)$ is not closed under any reasonable topology as depicted in Figure \ref{Pacman}, which motivates us to embed $L^1(\pared E, \R^+)$ into Radon measures in order to take this effect into account.    
\begin{figure}[h!]
\centering
\includegraphics[scale=0.6]{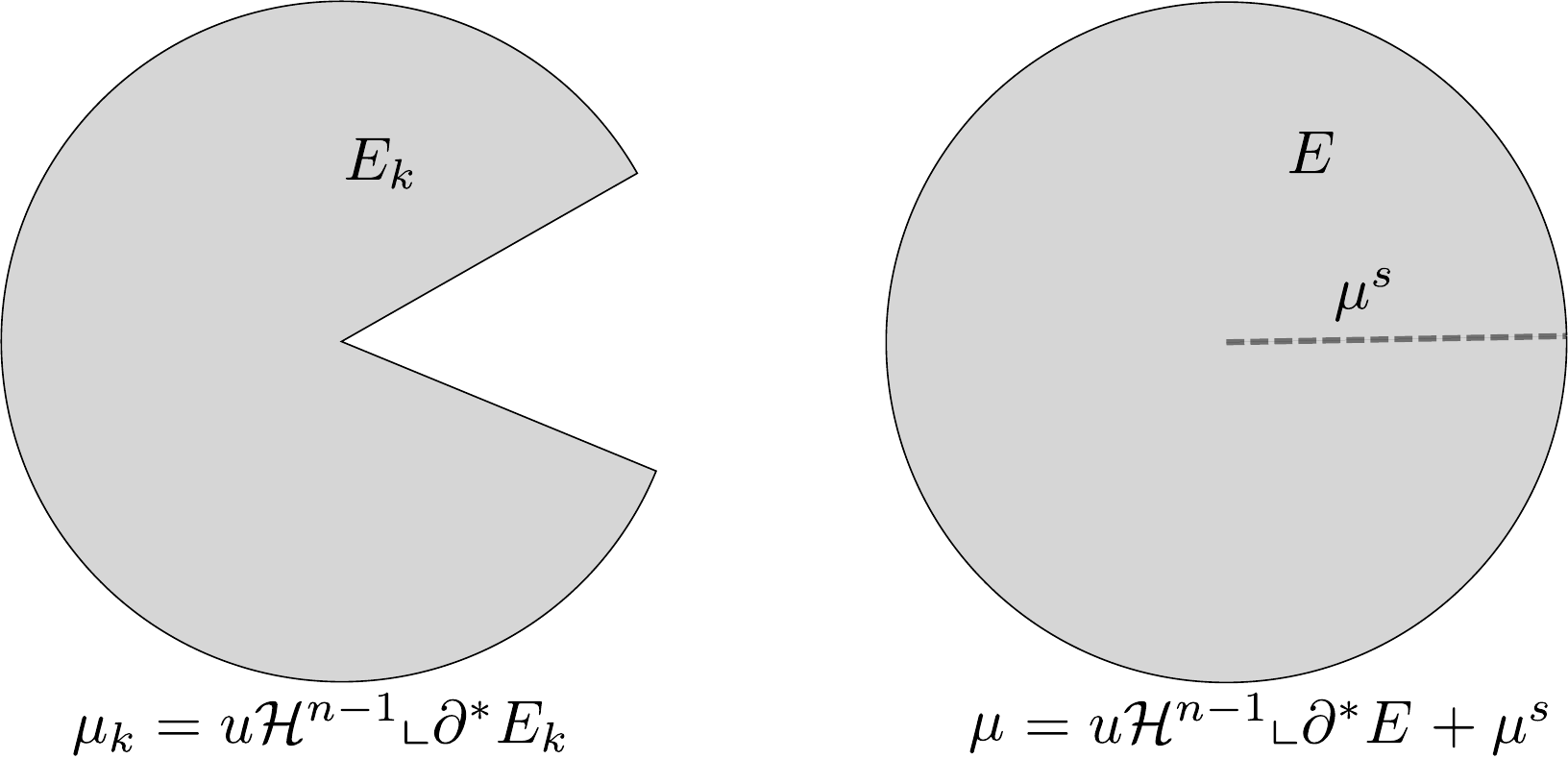}\caption{This example shows that we can easily escape from the class of couples $(E,u)$ with $u \in L^1(\pared E, \R^+)$}\label{Pacman}
\end{figure}


\subsection{Topology and necessary conditions for lower semicontinuity}
For every couple $(E,u)$ with $E$ a set of finite perimeter and $u \in L^1(\pared E, \R^+)$ a Borel function, let $\mu\in\mathcal M^+_{loc}(\R^n)$ be given by
\[
\mu:=u\,\H^{n-1}\restr\partial^* E = u |D\ca_E|.
\]
With this identification we can write
\[
\int_{\partial^*E} \psi(u)\,\d\H^{n-1}=\int_{\R^n} \psi\left(\frac{\d \mu}{\d|D \ca_E|}\right)\d |D\ca_E|\,.
\]
Fixed $m>0$, we consider the extension of $\F$ to the space
\[
\mathfrak S := \mathfrak C (\R^n)\times \mathcal{M}^+_{loc}(\R^n)\,, 
\]
as
\begin{equation*}
\F(E,\mu):=
\left\{
\begin{array}{ll}
\displaystyle\int_{\partial^*E} \psi(u)\,\d\H^{n-1} & \text{ if } \mu = u|D\ca_E| \text{ with } u\in L^1(\pared E,\R_+) \text{ Borel}\,,\\
&\\
+\infty & \text{ otherwise}\,.
\end{array}
\right.
\end{equation*}

\begin{remark}\label{rem:ac}
Couples $(E,u |D\ca_E|)\in\mathfrak{S}$ will be called \textit{absolutely continuous} couples and will be sometimes denoted by $(E,u)$ to simplify the notation.
\end{remark}

We are now in position to define our topology.

\begin{definition}\label{topology}
We endow $\mathfrak S$ with the product of the $L^1$ topology and the weak-* topology in $\mathcal{M}^+_{loc}(\R^n)$.
In particular, given $\left((E_k,\mu_k)\right)_{k\in\N}\subset \mathfrak S$ and  $(E,\mu)\in\mathfrak S$, we say that
\[
(E_k,\mu_k)\rightarrow (E,\mu) \ \ \ \text{in $\mathfrak S$}
\]
if and only if  $\ca_{E_k}\rightarrow \ca_E$ in $L^1$ and $\mu_k\wt \mu$ in $\mathcal{M}^+_{loc}(\R^n)$.
Moreover, we define the distance $d_{\mathfrak S}$ on $\mathfrak S$, which metrizes the above topology, as
\[
d_{\mathfrak S}\left[(E,\mu),(F,\nu)\right]:=
    \|\ca_E - \ca_F\|_{L^1}+\mathrm{d}_{\mathcal{M}}(\mu,\nu)\,,
\]
where $\mathrm{d}_{\mathcal{M}}$ is the distance given by Proposition \ref{prop:weakstarmetriz}. 
\end{definition}

In the sequel we will always use the above topology without mentioning it explicitly.
We now prove some necessary conditions that $\psi$ has to satisfy in order to ensure the lower semi-continuity of $\F$.
These conditions are in contrast with the superlinearity of the prototypes $\psi(u)=1+\gamma u^2$ used in \cite{Burger} (and with the classical assumption (A2a)).

\begin{figure}[h!]
\centering
\includegraphics[scale=0.6]{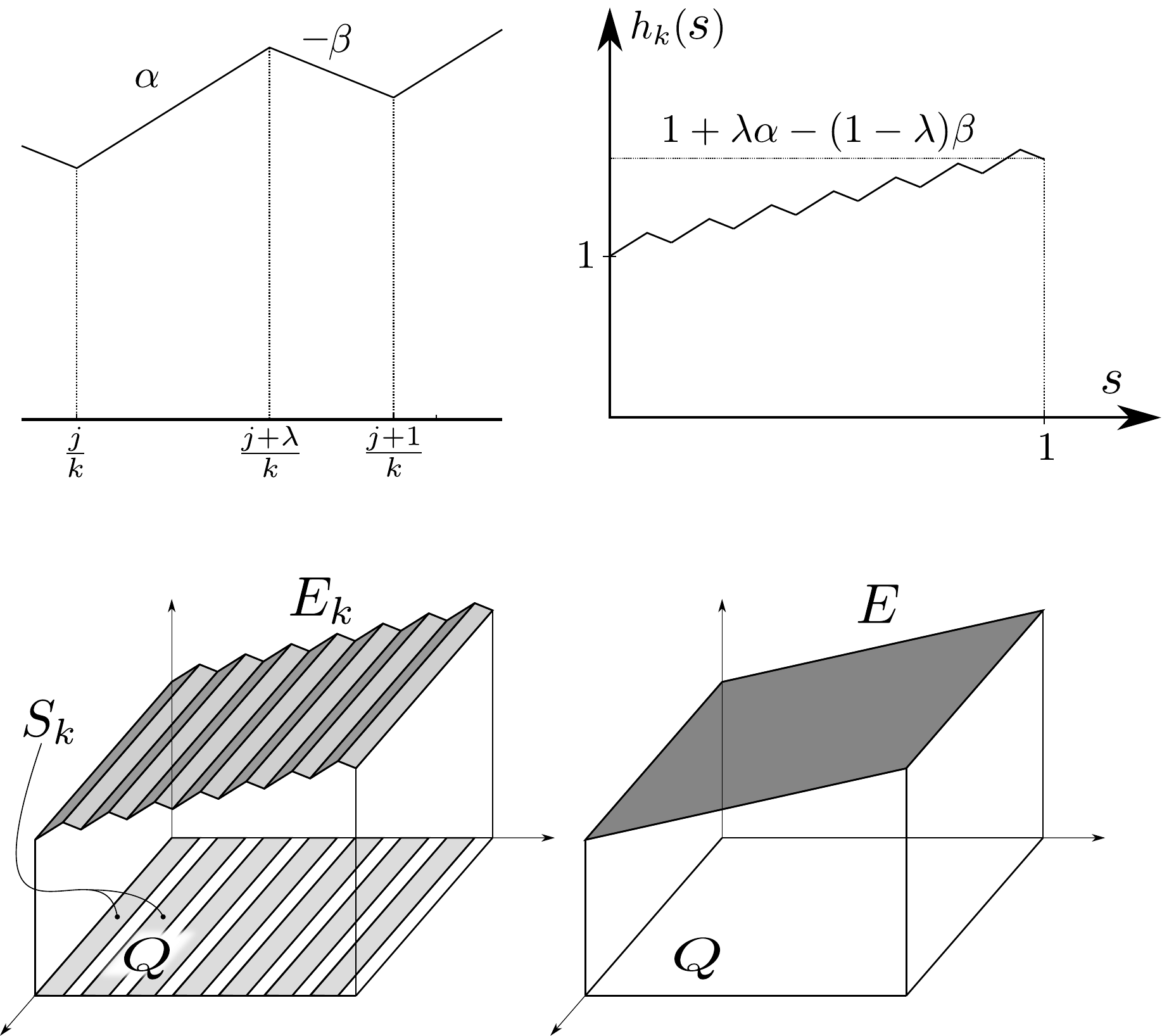}\caption{The set $E_k$ and its limit $E$ in $\R^3$. On $\pared E_k$ (on the left) we fix $u$ to be piecewise constant and equal to $a$ or $b$ in the upper part (depending on the different slopes of $h_k$) and $0$ everywhere else. The limit set $E$ (on the right) will have a piecewise constant $u$ as in \eqref{eqn: limit adatom counter} defined on $\pared E$.}\label{picture: counterexample}
\end{figure}

\begin{proposition}\label{prop:necessaryconditions}
Assume that $\F$ is lower semi-continuous.
Then, for all $a,b,\a,\b,\in \R^+$ and $0\leq \l \leq 1$, $\psi$ has to satisfy the relation
	\begin{equation}\label{eq:cond}
	\psi\left(\frac{a\l\sqrt{1+ \a^2} + b(1-\lambda)\sqrt{1+ \b^2}}{\sqrt{1+(\lambda\a - (1-\lambda)\b)^2}} \right) \leq \frac{\psi(a)\lambda \sqrt{1+ \a^2} + \psi(b)(1-\lambda)\sqrt{1+\b^2} }{\sqrt{1+(\lambda\a - (1-\lambda)\b)^2}}.
	\end{equation}
\end{proposition}

\begin{remark}\label{rem:equivconv}
\rm{ Relation \eqref{eq:cond} is obtained by testing $\F$ on a sequence of wriggled planes with a piecewise constant adatom density $u$ as illustrated in Figure \ref{picture: counterexample}.
}
\end{remark}

\begin{proof}[Proof of Proposition \ref{prop:necessaryconditions}]
Fix $0\leq \b \leq \a$, $0\leq \l\leq 1$ and, for every $k\in \N^{*}$, define the piecewise $\mathcal{C}^1$ function
$h_k:[0,1]\rightarrow\R$ as
	\begin{equation*}
	h_k(s):=\left\{
	\begin{array}{ll}
	s\a+1-\frac{(1-\l) j}{k}(\a+\b)& \text{if $s\in \left[\frac{j}{k}, \frac{j+\l}{k} \right]$},\\
	 & \\
	-s\b+1+\frac{\l (j+1)}{k}(\a+\b)& \text{if $s\in \left[\frac{j}{k}+\frac{\l}{k}, \frac{j+1}{k}\right]$}\,.
	\end{array}
	\right.
	\end{equation*}
Set 
	\[
	S_k:= \bigcup_{j=0}^{k-1} \left( \left[ \frac{j}{k},\frac{j+\l}{k} \right]\times \R^{n-2} \right), \ \ \ \ T_k:=S^c_k.
	\]
Let $Q:=[0,1]^{n-1}\subset \R^{n-1}$. For every $k\in \N^*$, define $H_k:Q\rightarrow \R$ as
	\[
	H_k(z):=h_k(z\cdot e_1)=h_k(z_1)\,,
	\]
where we write $z=(z_1,\dots,z_{n-1})\in\R^{n-1}$ and the set
	\begin{align*}
	E_k:=\left\{( z,s) \in Q\times\R \ | \ 0\leq s \leq H_k( z )\right\}.
	\end{align*}
Moreover we define the adatom density $u_k:\pared E_k\rightarrow[0,\infty)$ as
	\begin{equation*}
	u_k( x )=\left\{
	\begin{array}{ll}
 	a & \text{on  $ Gr\left(H_k, S_k \cap Q^{\circ} \right)$}\,,\\
 	b & \text{on $Gr\left(H_k, T_k \cap Q^{\circ} \right)$}\,,\\
	0 & \text{elsewhere}\,,
	\end{array}
	\right. 
	\end{equation*}
where, for any function $f:\R^{n-1} \rightarrow \R$ and for any $A\subset \R^{n-1}$,
	\[
	Gr(f,A):= \{ ( z , f( z ) ) \in \R^n \ | \   z \in A\}.
	\]
Let $\mu_k:= u_k|D \ca_{E_k}|$.\\

\noindent \textbf{Claim:} Up to extracting a subsequence (not relabeled), it holds that
\[
(E_k,\mu_k)\rightarrow(E,u |D\ca_E|)\,,
\]
where
	\[
	E:=\{ ( z ,s)\in Q\times \R \ | \ 0\leq s \leq H( z )  \}\,,
	\]
$H:Q\rightarrow\R$ is given by
\[
H( z ) :=  (\l \a- (1-\l)\b)( z \cdot  e_1)  +1
\]
 and $u:\partial^* E\rightarrow[0,\infty)$ is the adatom density
	\begin{equation}\label{eqn: limit adatom counter}
	u( x ):=\left\{
	\begin{array}{ll}
	\frac{\l a \sqrt{1+\a^2} + b(1-\l) \sqrt{1+\b^2} }{\sqrt{1 +  ( \l \a -(1-\l) \b)^2 )} } & \text{for  $z\in Gr\left(H, Q^{\circ} \right)$}\,,\\
	0 & \text{elsewhere}.
	\end{array}
	\right. \, 
	\end{equation}
\text{}\\
Moreover
\[
P(E_k; \pa Q \times \R) \rightarrow P(E; \pa Q \times \R)\,.
\]

Let us show how to derive the condition \eqref{eq:cond} assuming the validity of the claim. Notice that
\begin{align*}
 \F(E_k,\mu_k)&= \int_{\pared E_k} \psi(u_k ( x) ) \d \H^{n-1}( x )\\
&=\psi(0)\left[\H^{n-1}(Q)+P(E_k; \pa Q \times \R)\right]\\
&\hspace{2cm}+ \int_{\pared E_k \cap (Q^{\circ}\times (0,\infty)) } \psi(u_k ( x ) ) \d \H^{n-1}( x )\\
&= \psi(0)\left[\H^{n-1}(Q)+P(E_k; \pa Q \times \R)\right]\\
&\hspace{2cm}+ \int_{ Q^{\circ}\cap S_k } \psi(a ) \sqrt{1+|\nabla H_k( z )|^2}\d   z \\
&\hspace{4cm}+ \int_{ Q^{\circ}\cap T_k } \psi(b ) \sqrt{1+|\nabla H_k( z )|^2}\d   z \\
&=\psi(0)\left[\H^{n-1}(Q)+P(E_k; \pa Q \times \R)\right]+  \psi(a )\l \sqrt{1+\a^2}\\
&\hspace{2cm}+\psi(b ) (1-\l) \sqrt{1+\b^2}\,,
\end{align*}
where we used the identities
	\begin{equation*}
	\H^{n-1}(Q^{\circ}\cap S_k)=\H^{n-1}(Q\cap S_k)= \l, \ \ \ \H^{n-1}(Q^{\circ}\cap T_k)= \H^{n-1}(Q\cap T_k)= 1-\l.
	\end{equation*}
Analogously
	\begin{align*}
	\F(E, u) &=\psi(0)\left[\H^{n-1}(Q)+P(E_k; \pa Q \times \R)\right]\\
	&+ \psi\left(\frac{\l a \sqrt{1+\a^2} + b(1-\l) \sqrt{1+\b^2} }{\sqrt{1 +  ( \l \a -(1-\l) \b)^2 )} } \right)\sqrt{1 +  ( \l \a -(1-\l) \b)^2 )}.
	\end{align*}
By the semicontinuity of $\F$ and the fact that $P(E_k; \pa Q \times \R)\rightarrow P(E; \pa Q \times \R)$, we obtain \eqref{eq:cond}. We now focus in proving the claim. We divide the proof in two steps.\\
\text{}\\
\textbf{Step one:} \textit{$E_k\rightarrow E$ and $P(E_k;\pa Q\times \R) \rightarrow P(E;\pa Q\times \R)$}.
By the definition of $H_k$ and  $H$ we have
	\begin{equation}\label{unifconvgraphmarco}
	\sup_{ z \in Q}\{|H_k(  z ) -H( z ) |\} \leq \frac{C}{k}
	\end{equation}
for a constant $C$ depending on $\a,\b,\l$ only. In particular $H_k \rightarrow H$ in $C^0(Q)$ and thus $E_k \rightarrow E$.
Also by construction we obtain
	\begin{align*}
	\left|\, P(E_k; \pa Q \times \R)-P(E; \pa Q \times \R) \,\right|\leq\int_{\pa Q} |H_k(y) - H(y)|\d \H^{n-2}(  y )<\frac{C}{k}.
	\end{align*}
\text{}\\
\textbf{Step two:}  \textit{$\mu_k\wt u|D\ca_E|$.} Notice that 
	\[
	\mu_k(\R^n)< \max\{a,b\}P(E_k) < C
	\]
 for some constant $C>0$ and for some $R>0$. Thus, up to a subsequence (not relabeled), we can assume $\mu_k\wt\mu$ for some measure $\mu$. Moreover, by \eqref{unifconvgraphmarco} we have that $\mu(A)=0$ for all open sets $A\subset \R^{n}$ such that $|D \ca_E|(A)=0$. In particular, for $\H^{n-1}$-almost every $x\in \pa E$ the function 
 	\[
 	v( x ):= \lim_{r\rightarrow 0} \frac{\mu(B_r( x ))}{|D\ca_E| (B_r( x ))}
 	\]
turns out to be well defined. This implies that we can write
\[
\mu=v |D\ca_E|.
\]
It remains to show that $v=u$. By  \eqref{borelconv} and \eqref{foliation} we have, for all but countably many $r>0$,
 	\[
 	\mu(B_r)= \lim_{k\rightarrow\infty} \mu_k(B_r)\,.
 	\]
Fix $ \bar{x}  \notin Gr(H,Q^{\circ})$. Then, for $r$ small enough, we have that $\mu_k(B_r( \bar{x} ) ) =0$. Thus, $\mu(B_r( \bar{x} ) ) =0$, that implies $v(\bar{x})=0$ for all $\bar{x}\in\R^n\setminus Gr(H,Q^{\circ})$.

Let us now fix $ \bar{x}  \in Gr(H,Q^{\circ})$. For $r>0$ set
	\begin{align*}
	D_r&:=\{  z  \in Q^{\circ} \ | \ (  z ,H( z ))\in B_r(  x )\cap Gr(H,Q^{\circ})\},\\
	D_r^k&:=\{  z  \in Q^{\circ} \ | \ ( z ,H_k( z ))\in B_r( x )\cap Gr(H_k,Q^{\circ})\}
	\end{align*}
so that $B_r( \bar{x} )\cap \pa E=Gr(H,D_r)$, $B_r( \bar{x} )\cap \pa E_k=Gr(H_k,D_r^k)$. In particular,
\begin{align*}
 \mu_k(B_r( \bar{x} )) &=\int_{\pa E_k \cap B_r( \bar{x} )} u_k( x ) \d \H^{n-1}(  x )\\
  &=\int_{D^k_r} u_k(  z , H_k(z ) ) \sqrt{1+|\nabla H_k( z )|^2} \d  z \\
  &=a \sqrt{1+\a^2}\ \H^{n-1} (  D^k_r\cap  S_k  )  + b \sqrt{1+\b^2}\ \H^{n-1} (  D^k_r\cap  T_k  ).
 \end{align*}
Notice that, by \eqref{unifconvgraphmarco},
	\[
	\lim_{k\rightarrow +\infty} \H^{n-1} (  D_r^k\cap  S_k  )=\lambda \H^{n-1}(D_r)\,,
\]
and
\[
\lim_{k\rightarrow +\infty} \H^{n-1} (  D_r^k\cap  T_k  )=(1-\lambda) \H^{n-1}(D_r)\,.
\] 
Thus
 	\[
 	\mu(B_r( \bar{x} ) ) =a \sqrt{1+\a^2}\lambda \H^{n-1}(D_r) + b \sqrt{1+\b^2}(1-\l) \H^{n-1}(D_r).
 	\]
On the other hand, we have that
	\[
	|D\ca_E|(B_r(  \bar{x} ) ) =\int_{D_r}\sqrt{1+ | \nabla H(z)|^2} \d   z 
	=\H^{n-1}(D_r)\sqrt{1+(\l \a - (1-\l)\b)^2} .
	\]
Hence
	\[
	v(\bar{x}) = \frac{a\l \sqrt{1+\a^2} + b(1-\l) \sqrt{1+\b^2} }{\sqrt{1+(\l \a - (1-\l)\b)^2}}=u( \bar{x} )\,.
	\]
This proves the claim and thus concludes the proof.
\end{proof}

\begin{corollary}\label{cor:neccond}
If $\F$ is lower semi-continuous then $\psi$ is a convex function such that
	\begin{equation}\label{eq:cond3}
	\psi(a+b)\leq \psi(a)+\psi(b)\,,
	\end{equation}
for all $a,b\in \R_+$.
\end{corollary}

\begin{proof}
Take $\a=\b=0$ in \eqref{eq:cond} to deduce that $\psi$ is convex and set $\a=\b=\sqrt{3}$, $\lambda=\frac{1}{2}$ to obtain \eqref{eq:cond3}.
\end{proof}

The above result indicates that the conditions we are imposing so far on $\psi$ are, in general, not sufficient to ensure the lower semi-continuity of $\F$.  Moreover, even when $\psi$ is an admissible function, as in Definition \ref{def:energy}, and such that \eqref{eq:cond3} is satisfied, we do not expect $\F$ to be lower semi-continuous. Indeed, concentration phenomena can take place, as illustrated in Figure \ref{Pacman}, or along a sequence of shrinking balls with adatom density blowing up (see Remark \ref{degballs}). On the other hand, \eqref{eq:cond3} guarantees the finiteness of $ \lim_{s\to +\infty}  \psi(s)/s $. Taking all of this together into consideration, we build a candidate for the relaxed functional by replacing $\psi$ with its convex and subadditive envelope (see Section \ref{seccvxsubadd}) and by adding its \textit{recession function} on the singular part of the measure.


\begin{definition}
Given $\psi:\R_+ \rightarrow (0,\infty)$ be as in Definition \ref{def:energy}, let $\ov{\psi}$ be its \textit{convex subadditive envelope} (see Definition \ref{def:convexsubadditiveenvelope}), and set
	\[
	\Theta:=\lim_{s\rightarrow +\infty} \frac{\ov{\psi}(s) }{s}.
	\]
We define the functional $\ov{\F}:= \mathfrak S \rightarrow[0,\infty]$ as
	\begin{equation*}
	\ov{\F}(E,\mu):= \int_{\pared E} \ov{\psi}\left( u \right)\d\H^{n-1} + \Theta\mu^s(\R^n)\,,
	\end{equation*}
	where we write $\mu = u\H^{n-1}\restr\pared E + \mu^s$ using the Radon-Nikodym decomposition.
\end{definition}

\begin{remark}
Notice that, since the function $s\mapsto\ov{\psi}(s)/s$ is non increasing (see Lemma \ref{cvxsubaddparacd}), $\Theta$ in the above definition is well defined. Moreover, notice that $\ov{\F}(E,\mu)=\infty$ if and only if $\mu(\R^n)=\infty$.
Indeed, this follows from the inequalities
\[
\Theta s\leq \ov{\psi}(s)\leq \psi(0)+\Theta s\,,
\]
which, in turn, give us
\[
\Theta \mu(\R^n)\leq \ov{\F}(E,\mu)\leq \psi(0)P(E)+\Theta\mu(\R^n)\,.
\]
\end{remark}

The following result is a slight variation\footnote{Mainly we can remove the assumption of weak*-convergence of
$|D\ca_{E_k}|$ to $|D\ca_E|$ thanks to the subadditivity of $\ov{\psi}$.} of \cite[Theorem 2.34]{AFP}.
For the reader's convenience, we include here the proof adopting their notation.

\begin{theorem}\label{thm:lsc}
$\ov{\F}$ is lower semi-continuous.
\end{theorem}

\begin{proof}
Let $((E_k,\mu_k))_{k\in \N} \subset \mathfrak S$ be a sequence converging to $(E,\mu)$ in $\mathfrak S$, that is
	$\ca_{E_k} \rightarrow \ca_E$ in $L^1$ and $ \mu_k\wt \mu$ in $\mathcal{M}^+_{loc}(\R^n)$. Let
	\begin{align*}
	\mu_k&= u_k \H^{n-1}\restr \pared E_k+\mu_k^s,\qquad \quad \mu=u \H^{n-1}\restr \pared E +\mu^s.
	\end{align*}
In view of the characterization of $\ov{\psi}$ (see Lemma \ref{lem:supconv}), there exist families of real numbers $\{a_j\}_{j\in \N}$, $\{b_j\}_{j\in \N}$ with $a_j, b_j\geq 0$ and such that
	\[
	\ov{\psi}(s):= \sup_{j\in \N} \{a_j s+ b_j\},\ \qquad \quad \Theta = \sup_{j\in \N} \{a_j\}\,.
	\]
Consider $A_1,\ldots,A_m$ pairwise disjoint open, bounded subsets of $\R^n$. For any $g_j\in \mathcal{C}^1_c(A_j)$, with $0 \leq g_j\leq 1$, we have
	\begin{align*}
	&\int_{A_j\cap \pared E_k} \ov{\psi}(u_k) \d \H^{n-1}  + \Theta \mu_k^s(\R^n)\\
	&\hspace{1cm} \geq \int_{A_j\cap \pared E_k} g_j(a_j  u_k + b_j  ) \d \H^{n-1} + \int_{A_j} g_j a_j \d \mu_k^s\\
	&\hspace{1cm}=\int_{A_j \cap \pared E_k} g_j  a_j  u_k \d \H^{n-1}  + \int_{A_j\cap \pared E_k}g_j b_j  \d \H^{n-1} + \int_{A_j} g_j a_j \d \mu_k^s\\
	&\hspace{1cm}=\int_{A_j} g_j  a_j  \d \mu_k  + \int_{A_j\cap \pared E_k}g_j b_j  \d \H^{n-1}.
	\end{align*}
Adding with respect to $j$, we obtain
	\begin{align*}
	\ov{\F}(E_k,\mu_k) \geq \sum_{j=1}^m\int_{A_j} g_j  a_j  \d \mu_k  + \int_{A_j\cap \pared E_k}g_j b_j  \d \H^{n-1}.
	\end{align*}
Since $b_j \geq 0$ and $ \langle|D\ca_E|,g_j \rangle \leq \liminf_k \langle|D\ca_{E_k}|, g_j \rangle$ for all $j$ (here $\langle\cdot,\cdot\rangle$ is the duality pairing), taking the liminf we get 
	\begin{align}
	 \liminf_{k\rightarrow +\infty} \ov{\F}(E_k,\mu_k) &\geq \sum_{j=0}^m \int_{A_j } g_j  a_j  \d \mu  + \int_{A_j\cap \pared E} g_j b_j  \d \H^{n-1} \nonumber \\
	 &=\sum_{j=0}^m \int_{A_j \cap \pared E} g_j  (a_j u +b_j) \d\H^{n-1}  + \int_{A_j} g_j a_j  \d \mu^s \label{eqn: yeah!}.
	\end{align}
Let $N$ be a $|D\ca_E|-$negligible set on which $\mu^s$ is concentrated, and define the functions $\varphi_j:\R^n\rightarrow\R$ and $\varphi:\R^n\rightarrow\R$ as
	\begin{equation*}
	\varphi_j(x):=\left\{
	\begin{array}{ll}
	a_j u(x) + b_j & \text{for $x\in \pared E \setminus N$}\,,\\
	a_j &  \text{for $x\in N$}\,,
	\end{array}	
 	\right.
 	\end{equation*}
 	\begin{equation*}
 		\varphi(x):=\left\{
	\begin{array}{ll}
	\ov{\psi}(u(x) ) & \text{for $x\in \pared E\setminus N$}\,,\\
	\Theta &  \text{for  $x\in N$}\,,
	\end{array}	
 	\right.
	\end{equation*}
and set $\nu:= |D\ca_E| + \mu^s$. With this notation, equation \eqref{eqn: yeah!} can be written as
	\begin{align*}
	 \liminf_{k\rightarrow +\infty} \ov{\F}(E_k,\mu_k) &\geq\sum_{j=0}^m \int_{A_j} g_j \varphi_j \d \nu.
	\end{align*}
Taking the supremum among all the $g_j\in \mathcal{C}^1_c(A_j)$ with $0\leq g_j \leq 1$, we get (since $\varphi_j\geq 0$ for all $j$) 
	\begin{align*}
	 \liminf_{k\rightarrow +\infty} \ov{\F}(E_k,\mu_k) &\geq\sum_{j=0}^m \int_{A_j}   \varphi_j \d \nu.
	\end{align*}
By \cite[Lemma 2.35]{AFP}, we have that
	\[
	\int_{\R^n} \sup_j \{\varphi_j\} \d \nu = \sup\left\{ \sum_{j\in J} \int_{A_j} \varphi_j \d \nu \right\} 
	\]
where the supremum ranges over all finite sets $J\subset \N$ and all families of pairwise disjoint open and bounded sets $A_j\subset \R^ n$. Thus, we conclude that
	\begin{align*}
	 \liminf_{k\rightarrow +\infty} \ov{\F}(E_k,\mu_k) &\geq  \int_{\R^n}  \sup_j \{\varphi_j\} \d \nu= \int_{\R^n} \varphi \d \nu \\
	 & =  \int_{\pared E}  \ov{\psi}(u(x))  \d \H^{n-1} + \Theta\mu^s(\R^n) =\ov{\F} (E,\mu).
	\end{align*}
\end{proof}


\subsection{The relaxed functional}

We start by recalling the notion of \textit{relaxation} of a functional.
We refer to \cite{DM} and \cite{Bra} for a treatment of $\Gamma$-convergence.

\begin{definition}\label{def:relaxgen}
Let $(X,\tau)$ be a topological space and let $F:X\rightarrow[-\infty,+\infty]$. 
We define the \emph{relaxed functional} $\overline F:X\to [-\infty,+\infty]$ of $F$ as the largest lower semi-continuous functional
$G:X\rightarrow[-\infty,+\infty]$ such that $G\leq F$.
\end{definition}

The following characterization of the relaxed functional holds true.

\begin{proposition}
Let $(X,d)$ be a metric space and let $F:X\rightarrow[-\infty,+\infty]$.
Then, the relaxed functional $\overline F:X\to [-\infty,+\infty] $ of $F$
is characterized by the following two conditions:
\begin{enumerate}[i)]
\item (Liminf inequality) for every $x \in X$ and every sequence $(x_k)_{k\in  \mathbb N}$ such that $x_k \to x$, $$\overline F(x) \leq \liminf_{k \to \infty} F(x_k).$$
\item (Recovery sequences) for every $x \in X$ there exists a sequence $(x_k)_{k\in  \mathbb N}$ such that $x_k \to x$ and $$\limsup_{k \to \infty} F(x_k)\leq \overline F(x).$$
\end{enumerate}
\end{proposition}

We now prove the main theorem of this section.

\begin{theorem}\label{thm:relax}
The functional $\ov{\F}$ is the relaxation of $\F$.
To be precise, the following hold:
\begin{itemize}
\item[(i)] for every $(E,\mu)\in \mathfrak S$ and every sequence
$\left((E_k,\mu_k)\right)_{k\in\N}\subset \mathfrak S$ such that $(E_k,\mu_k)\rightarrow (E,\mu)$,
we have that
\[
\ov{\F}(E,\mu)\leq\liminf_{k\rightarrow\infty}\F(E_k,\mu_k)\,,
\]
\item[(ii)] for every $(E,\mu)\in \mathfrak S$ there exists $\left((E_k,\mu_k)\right)_{k\in\N}\subset \mathfrak S$
with $(E_k,\mu_k)\rightarrow (E,\mu)$ such that
\[
\limsup_{k\rightarrow\infty}\F(E_k,\mu_k)\leq \ov{\F}(E,\mu)\,.
\]
\end{itemize}
\end{theorem}

The proof of the above theorem is long and will be divided into several steps. Let us first sketch it briefly. The liminf inequality will be a consequence of Theorem \ref{thm:lsc} and the fact that $\ov{\psi}\leq\psi$. In order to construct recovery sequences, the case $\psi=\ov{\psi}$ will be easier to deal with so let us assume here that there exists $s_0\in(0,\infty)$ such that $\psi=\ov{\psi}$ in $[0,s_0]$ and $\ov{\psi}<\psi$ in $(s_0,\infty)$ (see Remark \ref{recipe}). We will approximate the two terms of $\ov{\F}$ separately. To explain how we deal with the first one, for the sake of simplicity let us consider a smooth set $E\subset\R^n$ and a constant adatom density $u\equiv c > x_0$. We construct a recovery sequence $((E_k,u_k))_{k\in\N}\in \mathfrak S$ as follows:
write $c=r s_0$ for some $r>1$. Then, since $\ov{\psi}$ is \emph{linear} in $[s_0,\infty)$,
we have
\[
\ov{\psi}(c)=\ov{\psi}(r s_0)=r\ov{\psi}(s_0)=r\psi(s_0)\,.
\]
Therefore take $u_k\equiv s_0$ and we let $(E_k)_{k\in\N}$ be a sequence of smooth sets converging to $E$ in $L^1$ and such that
\[
\H^{n-1}(\partial E_k)\rightarrow r\H^{n-1}(\partial E)\,.
\] 
This will be done by a wriggling process (Lemma \ref{lem:wrigg}) similiar to the one pictured in Figure \ref{pic:wriggle} for the unit circle.

\begin{figure}[h!]
\begin{center}
\includegraphics[scale=0.50]{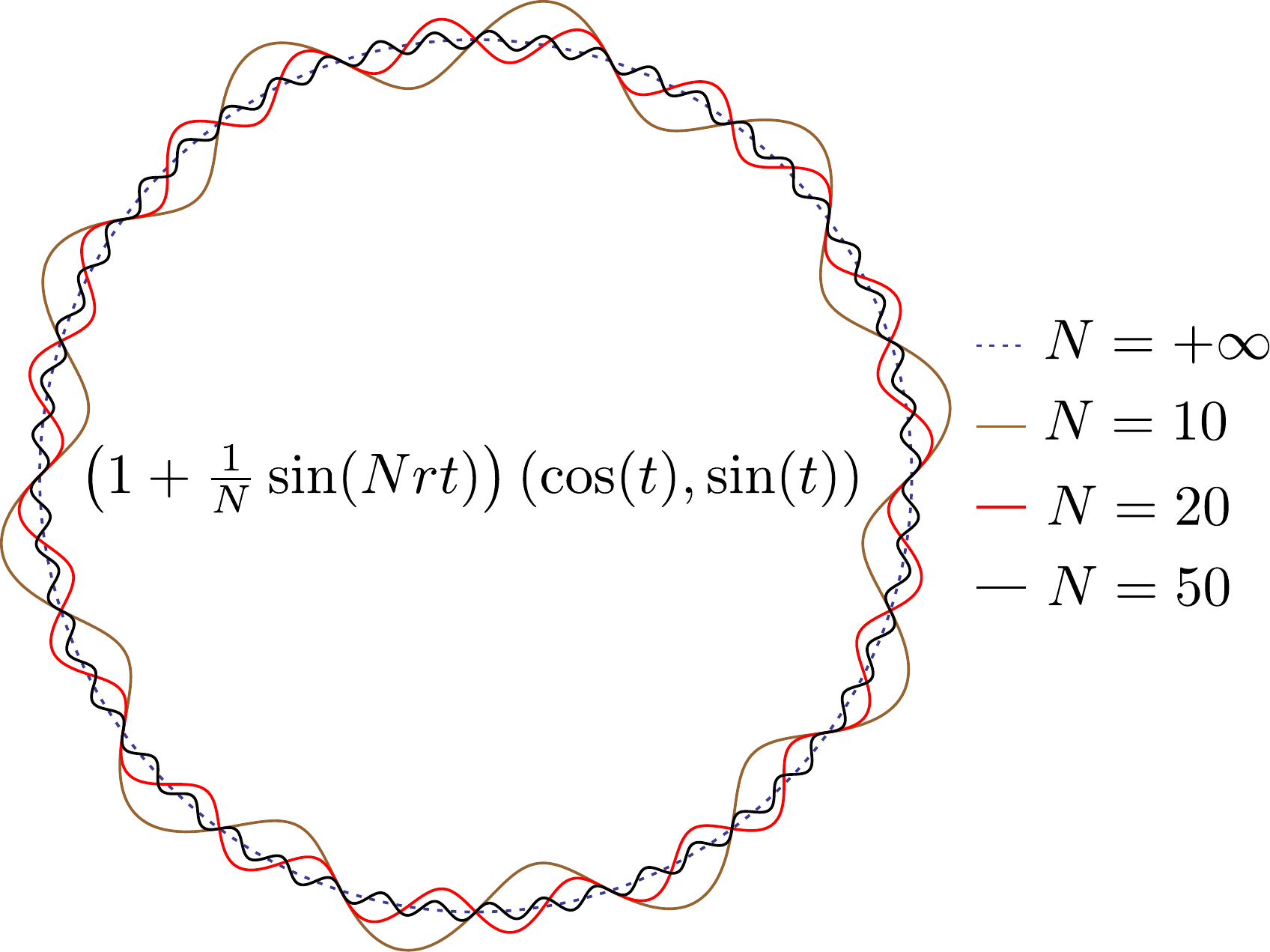}  
\caption{Approaching the unit circle by curves with constant but bigger perimeter.
Notice that the recovery sequence here exhibits features similar to numerical simulations of the evolution equation in \cite{RV2}.}
\label{pic:wriggle}
\end{center}
\end{figure}

To treat the second term we are led by the following observation: a couple $(\varnothing, \delta_0)$ can be recovered by shrinking spheres with increasing adatom density. This, combined with the fact that any $\mu^s$ can be approximated by a sum of such Dirac deltas and with a suitable mollification argument, will allow us to recover any $(\varnothing,\mu^s)$ (see Proposition \ref{propo singular}).
In a last step, we show that we can combine these two approximations to get close to any such $(E,\mu)$ as much as we want.

We now prove a density result in $\mathfrak S$ allowing us to restrict the analysis to the above scenario.

\begin{proposition}\label{prop:density}
Let $(E,u)\in \mathfrak S$. Then, there exists a sequence of bounded smooth sets $(E_k)_{k\in\N}$ and a sequence of functions $(u_k)_{k\in\N}$ with $u_k\in L^1(\partial E_k,\R_+)$ Borel, with the following properties:
\begin{itemize}
\item[(i)] for every $k\in\mathbb{N}$ there exists a family $(M^k_i)_{i\in\N}\subset\partial E_k$ of smooth manifolds with Lipschitz boundary, with
$\mathcal{H}^{n-1}\left(\partial E_k\setminus \bigcup_{i\in\N}M^k_i\right)=0$,
such that $u_k$ is constant on each $M^k_i$, for every $i\in\N$,
\item[(ii)] $\ca_{E_k}\rightarrow \ca_E$ in $L^1$, and $|D\ca_{E_k}|\wt|D\ca_E|$ as $k\to\infty$,
\item[(iii)] setting $\mu_k:=u_k \H^{n-1}\restr\partial E_k$ and $\mu:=u \H^{n-1}\restr\partial^* E$, we get $\mu_k\wt\mu$ and $\mu_k(\R^n)\rightarrow\mu(\R^n)$ as $k\to\infty$,
\item[(iv)] $\ov\F(E_k,u_k)\rightarrow\ov\F(E,u)$ as $k\to\infty$.
\end{itemize}
\end{proposition}

\begin{proof}
\textbf{Step one:} \emph{approximation of a bounded set.} Assume that $E$ is bounded and let $Q\subset\R^n$ be a closed cube
with edges of length $L$ parallel to the coordinate axes such that $E\subset Q$. By a standard argument (see \cite[Theorem 3.42]{AFP}), it is possible to construct a sequence of bounded smooth sets $(E_k)_{k\in\N}$ with $E_k\Subset Q$ such that
\begin{equation}\label{eq:conv}
\ca_{E_k}\rightarrow \ca_E\quad \text{in } L^1\,,\quad\quad|D\ca_{E_k}|\wt|D\ca_E|\,,\quad\quad P(E_k)\rightarrow P(E)\,.
\end{equation}
For every $k\in\N$, write
\[
Q=\bigcup_{j=1}^{k^n} Q^k_i\,,
\]
where each $Q^k_j$ is a closed cube of side $2L/k$ with edges parallel to the coordinate axes.
By \cite{Falconer}, up to an arbitrarily small rotation of the $E_k$'s and of $E$, it is possible to assume that
\begin{equation}\label{eq:null}
\H^{n-1}\left(\, \partial E_k \cap \bigcup_{j=1}^{k^n}\partial Q^k_j \,\right)=0\,,\quad\quad
\H^{n-1}\left(\, \partial^* E \cap \bigcup_{j=1}^{k^n}\partial Q^k_j \,\right)=0\,
\end{equation}
for every $k\in\N$.
Notice that $\partial E_k\cap (Q_j^k)^\circ$, where $(Q_j^k)^\circ$ denotes the open cube, is made by at most countably many smooth manifolds with Lipschitz boundary. Call them $(M^k_i)_{i\in\N}$.
By using \eqref{eq:conv}, together with \eqref{eq:null}, up to a subsequence of the $E_k$'s, it is also possible to assume that
\begin{equation}\label{eq:small}
\sum_{j\in I_k}\left|\,\frac{\H^{n-1}(\partial E_k\cap Q^k_j)}{\H^{n-1}(\partial^* E\cap Q^k_j)} -1 \,\right|<\frac{1}{k}\,,\quad\quad\quad
\sum_{j\in J_k}\H^{n-1}(\partial E_k\cap Q^k_j)<\frac{1}{k}\,,
\end{equation}
where we set
\[
I_k:=\{\, j\in\{1,\dots, k^n\} \,:\, \H^{n-1}(\partial^* E\cap Q^k_j)\neq0  \,\}\,
\]
and
\[
J_k:=\{\, j\in\{1,\dots, k^n\} \,:\, \H^{n-1}(\partial^* E\cap Q^k_j)=0  \,\}\,.
\]
Finally, let us define the function $u_k:\partial E_k\rightarrow\R$ as
\begin{equation}\label{locaverage}
u_k(x):=\fint_{\partial^* E\cap Q_j^k}u\,\H^{n-1}=\frac{1}{\H^{n-1}(\partial^* E\cap Q_j^k)}\int_{\partial^* E\cap Q_j^k}u\,\H^{n-1}\,,
\end{equation}
if $x\in \partial E_k\cap (Q_j^k)^\circ$, with $j\in I_k$, and $u_k(x):=0$ otherwise.
Notice that $u_k$ is not defined only on a set of $\H^{n-1}$ measure zero.

Let $\mu_k:=u_k\,\H^{n-1}\restr\partial E_k$ and $\mu:=u\,\H^{n-1}\restr\partial^* E$.
We want to prove that $\mu_k\wt\mu$. Take $\varphi\in \mathcal C_c(\R^n)$ and fix $\delta>0$. Using the uniform continuity of $\varphi$, it is possible to find $\bar{k}\in\N$ such that, for every $k\geq\bar{k}$, it holds $|\varphi(x)-\varphi(y)|<\delta$ whenever $x,y\in Q^k_j$ and
for every $j=1,\dots,k^n$. Let us denote by $x^k_j$ the center of the cube $Q^k_j$.
Then we have that
\begin{align}\label{eq:vicu}
&\left|\, \int_{\partial E_k}\varphi u_k\,\d\H^{n-1} - \int_{\partial^* E}\varphi u\,\d\H^{n-1}  \,\right| \nonumber\\
&\hspace{0.5cm}\leq \sum_{j=1}^{k^n}
\left|\, \int_{\partial E_k\cap Q^k_j}\varphi u_k\,\d\H^{n-1} - \int_{\partial^* E\cap Q^k_j}\varphi u\,\d\H^{n-1}  \,\right|
    \nonumber \\
&\hspace{0.5cm}=\sum_{j\in I_k}
\left|\, \int_{\partial E_k\cap Q^k_j}\varphi u_k\,\d\H^{n-1} - \int_{\partial^* E\cap Q^k_j}\varphi u\,\d\H^{n-1}  \,\right|
    \nonumber \\
&\hspace{0.5cm}= \sum_{j\in I_k}\left|\, \left( \fint_{\partial^* E\cap Q^k_j}u\,\d\H^{n-1} \,\right)
    \left(\, \int_{\partial E_k\cap Q^k_j}\varphi\,\d\H^{n-1}\,\right)
    - \int_{\partial^* E\cap Q^k_j}\varphi u\,\d\H^{n-1}  \,\right| \nonumber \\
&\hspace{0.5cm}= \sum_{j\in I_k}\Biggl|\, \left( \fint_{\partial^* E\cap Q^k_j}u\,\d\H^{n-1} \,\right)
    \left(\, \int_{\partial E_k\cap Q^k_j} (\varphi - \varphi(x^k_j))\,\d\H^{n-1} \right)\nonumber \\
&\hspace{2cm}+\left( \fint_{\partial^* E\cap Q^k_j}u\,\d\H^{n-1} \,\right)\varphi(x^k_j)\H^{n-1}(\partial E_k\cap Q^k_j)\, \nonumber \\
&\hspace{3cm} - \int_{\partial^* E\cap Q^k_j}(\varphi - \varphi(x^k_j)) u\,\d\H^{n-1}
    -  \varphi(x^k_j)\int_{\partial^* E\cap Q^k_j} u\,\d\H^{n-1}\,\Biggr| \nonumber \\
&\hspace{0.3cm}\leq \sum_{j\in I_k} \left( \int_{\partial^* E\cap Q^k_j}u\,\d\H^{n-1} \,\right) \Bigg[\,
	\left(\frac{\H^{n-1}(\partial E_k\cap Q^k_j)}{\H^{n-1}(\partial^* E\cap Q^k_j)}+1 \,\right)\delta \nonumber \\
&\hspace{6cm}+\left|\frac{\H^{n-1}(\partial E_k\cap Q^k_j)}{\H^{n-1}(\partial^* E\cap Q^k_j)}-1 \,\right||\varphi(x_j^k)| \,\Bigg] \nonumber \\
&\hspace{0.3cm}\leq \left(2\delta+\frac{\delta+\sup|\varphi|}{k} \right)\|u\|_{L^1(\partial^* E)}\,,
\end{align}
where in the first step we used \eqref{eq:null} and in the last one the first condition in \eqref{eq:small}.
Letting $k\rightarrow\infty$ we get that
\[
\left|\, \int_{\partial E_k}\varphi u_k\,\d\H^{n-1} - \int_{\partial^* E}\varphi u\,\d\H^{n-1}  \,\right| \to 0\,.
\]
Since $\varphi \in \mathcal C_c(\R^n)$ is arbitrary we conclude that $\mu_k\wt\mu$.
Moreover, by taking $\varphi\in\mathcal{C}_c(\R^n)$ such that $\varphi\equiv1$ in $Q$, we have that $\mu_k(\R^n)\rightarrow\mu(\R^n)$.

Finally, we claim that $\ov{\F}(E_k,u_k)\rightarrow\ov{\F}(E,u)$ as $k\rightarrow\infty$. Indeed,
\begin{align*}
\ov{\F}(E_k,u_k)-\ov{\F}(E,u)&=\int_{\partial E_k} \ov{\psi}(u_k)\,\d\H^{n-1} - \int_{\partial^* E} \ov{\psi}(u)\,\d\H^{n-1} \\
&= \sum_{j\in I_k}
    \int_{\partial E_k\cap Q^k_j} \ov{\psi}(u_k)\,\d\H^{n-1}-\int_{\partial^* E\cap Q^k_j} \ov{\psi}(u)\,\d\H^{n-1}\\
&\hspace{2cm}+\psi(0)\sum_{j\in J_k} \H^{n-1}(\partial E_k\cap Q^k_j)\\
&\leq \sum_{j\in I_k}
    \left(\frac{\H^{n-1}(\partial E_k\cap Q^k_j)}{\H^{n-1}(\partial^* E\cap Q^k_j)}-1 \,\right)\,
    \int_{\partial^* E\cap Q^k_j} \ov{\psi}(u)\,\d\H^{n-1}\\
&\hspace{2cm}+\psi(0)\sum_{j\in J_k} \H^{n-1}(\partial E_k\cap Q^k_j)\\
&\leq \frac{\psi(0)(1+P(E))+\Theta\|u\|_{L^1(\partial^* E)}}{k}\,,
\end{align*}
where in the second step we used Jensen's inequality, while in the last one we invoked \eqref{eq:small} and the fact that
$\ov{\psi}(u)\leq \psi(0)+\Theta u$. Letting $k\rightarrow\infty$ we get that
\[
\limsup_{k\to\infty}\ov{\F}(E_k,u_k)\leq \ov{\F}(E,u)\,.
\]
The other inequality follows from the lower semi-continuity of the functional $\ov{\F}$ (see Theorem \ref{thm:lsc}).
We thus conclude the proof of this step.\\

\textbf{Step two:} \emph{reduction to bounded sets.} Let $E$ be a set of finite perimeter, and assume that $E$ is not bounded.
Using the coarea formula (see \cite[Theorem 2.93]{AFP}), for every $k\in\N$ it is possible to find a sequence
$(R_k)_{k\in\N}$ with $R_k\nearrow\infty$, such that $F_k:=E\cap B_{R_k}(0)$ satisfies
\[
\left\|\ca_{F_k}-\ca_{E}\right\|_{L^1}<\frac{1}{2k}\,,
\quad\quad\quad
P(F_k)=P(E,B_{R_k}(0))+\H^{n-1}(\partial B_{R_k}(0)\cap E)\,,
\]
with $\H^{n-1}(\partial B_{R_k}(0)\cap E)<1/2k$.
Moreover, extracting if necessary a (not relabeled) subsequence , we can also assume that
\[
\int_{\partial^* E\setminus B_{R_k}(0)} u\,\d\H^{n-1}<\frac{1}{2k}\,.
\]
Define $\widetilde{u}_k:\partial^* F_k\rightarrow\R$ as
\[
\widetilde{u}_k(x):=
\left\{
\begin{array}{ll}
u(x) & \text{ if } x\in\partial^* E\cap B_{R}(0) \,,\\ 
0 & \text{ otherwise }\,.
\end{array}
\right.
\]
Then
\begin{align*}
|\ov{\F}(F_k,\widetilde{u}_k)-\ov{\F}(E,u)|&=\left|\, \int_{\partial^* E}\ov{\psi}(u)\,\d\H^{n-1}
    -\int_{\partial^* F_k}\ov{\psi}(\widetilde{u}_k)\,\d\H^{n-1}  \,\right|\\
&=\left|\, \int_{\partial B_{R_k}\cap E}\psi(0)\,\d\H^{n-1}
    +\int_{\partial^* E\setminus B_{R_k}(0)} \ov{\psi}(u)\,\d\H^{n-1} \,\right|   \\
&\leq\H^{n-1}(\partial B_{R_k}\cap E)\psi(0)+\int_{\partial^* E\setminus B_{R_k}(0)} \ov{\psi}(u)\,\d\H^{n-1}\\
&\leq \frac{2\psi(0)+\Theta}{2k}\,,
\end{align*}
where in the last step we used again the fact that $\ov{\psi}(u)\leq \psi(0)+\Theta u$. Moreover, for every $\varphi\in\mathcal{C}_c(\R^n)$, we have
\begin{align}\label{eq:convapproxmeas}
\left|\, \int_{\partial^* E}\varphi u\,\d\H^{n-1}-\int_{\partial^* F_k} \varphi \widetilde{u}_k\,\d\H^{n-1}  \,\right|
    &=\left|\, \int_{\partial^* E\setminus B_{R_k}(0)} \varphi u\,\d\H^{n-1} \,\right| \nonumber \\
&\leq \frac{\sup|\varphi|}{2k}\,.
\end{align}
Set $\widetilde{\mu}_k:= \widetilde{u}_k\H^{n-1}\restr\partial^* F_k$ and $\mu:= u\H^{n-1}\restr\partial^* E$.
Up to a (not relabeled) subsequence, we can assume that $\mathrm{d}_{\mathcal{M}}(\widetilde{\mu}_k,\mu)\leq 1/2k$.
In particular, \eqref{eq:convapproxmeas} gives us that $\widetilde{\mu}(\R^n)\rightarrow\mu(\R^n)$.
Now, by Step one, for every $k\in\N$ let $(E_k,u_k)\in \mathfrak S$, with $E_k$ smooth and bounded, be such that
\[
\|\ca_{E_k}-\ca_{F_k}\|_{L^1}<\frac{1}{2k}\,,\quad
\mathrm{d}_{\mathcal{M}}(\widetilde{\mu}_k,\mu_k)\leq\frac{1}{2k}\,,\quad
|\ov{\F}(F_k,\widetilde{u}_k)-\ov{\F}(E_k,u_k)|\leq\frac{1}{2k}\,,
\]
where $\mu_k:=u_k\H^{n-1}\restr\partial E_k$. Moreover, $\mu_k(\R^n)\rightarrow\mu(\R^n)$.
So,the sequence $((E_k,u_k))_{k\in\N}$ satisfies the requirements of the lemma.
\end{proof}

We now carry on the wriggling construction. The idea is to wriggle by a suitable factor each piece $M^k_i$ where $u_k$ is constant, staying in a small tubular neighborhood and leaving its boundary untouched, so that we can glue all the pieces together afterwards.

\begin{lemma}\label{lem:wrigg}
Let $M\subset\R^n$ be a bounded smooth $(n-1)$-dimensional manifold having Lipschitz boundary such that
$\H^{n-1}(M)<\infty$, and let $r\geq1$.
Then, there exist a sequence of smooth $(n-1)$-dimensional manifolds $(N_k)_{k\in\N}$ such that
\[
\partial N_k=\partial M\,,\quad\quad N_k\subset(M)_{1/k}\,,\quad\quad \H^{n-1}(N_k)\rightarrow r\H^{n-1}(M)\,,
\]
where $(M)_{1/k}:=\{\, x\in\R^n \,:\, \mathrm{d}(x,M)<1/k \,\}$ and $\mathrm{d}(x,M):=\inf\{\, |x-y| \,:\, y\in M \,\}$.
\end{lemma}

\begin{proof}
If $r=1$, it suffices to set $N_k = M$. Assume $r > 1$.
For $k\in\N^*$, let $C_k\subset M$ be a compact set such that $M\setminus C_k\subset (\partial M)_{1/k}$ and let
$\varphi_k\in C^\infty_c(M)$ be such that
\begin{equation}\label{eq:varphi}
0\leq\varphi_k\leq1\,,\quad\quad \varphi_k\equiv1 \text{ on } C_k\,,\quad\quad
|\nabla^M\varphi_k|\leq Ck\,,
\end{equation}
for some constant $C>0$. In the sequel, $\tau_1(x),\dots,\tau_{n-1}(x)$ will denote an orthonormal base of the tangent space of $M$ at a point $x\in M$.
Fix a point $\bar{x}\in M$ and let $v\in\R^n$ be such that
\begin{equation*}
0< \sum_{i=1}^{n-1}(v\cdot\tau_i(\bar{x}))^2\,,\quad\quad |\bar{x}\cdot v|<\frac{\pi}{2}\,.
\end{equation*}
We claim that it is possible to find a sequence $(t_k)_{k\in\N}$ such that
\begin{equation}\label{eq:rper}
\int_M\sqrt{1+\frac{t^2_k}{k^2}\cos^2(t_k(x\cdot v)) \sum_{i=1}^{n-1}\left(\tau_i(x)\cdot v\right)^2}\,\d\H^{n-1}(x)
    =r\H^{n-1}(M)\,.
\end{equation}
Indeed, by continuity it is possible to find $\lambda,\varepsilon>0$ such that
\begin{equation}\label{eq:estg}
\H^{n-1}(G)=\lambda
\end{equation}
where
\begin{equation}\label{eq:definitiong}
G:=\left\{\, x\in M \,:\, \varepsilon<\sum_{i=1}^{n-1}(v\cdot\tau_i(x))^2\,,\quad
    |x\cdot v| <\frac{\pi}{2}-\varepsilon \,\right\}\,.
\end{equation}
For every $t>0$ define
\[
Z_t:=\left\{\, x\in M \,:\, t|x\cdot v| \,\,\mathrm{mod }\  \pi \in \left(\frac{\pi}{2}-\varepsilon ,\frac{\pi}{2}+\varepsilon \right) \,\right\}\,,
\]
and notice that
\begin{equation}\label{eq:estz}
\liminf_{t\rightarrow\infty}\H^{n-1}(G \setminus Z_t)\geq\frac{\lambda}{2}\,.
\end{equation}
Let $\delta := \cos(\pi/2-\varepsilon) > 0$. By using \eqref{eq:estg}, \eqref{eq:definitiong} and \eqref{eq:estz}, we have that
\begin{align*}
&\liminf_{t\rightarrow\infty}
    \int_M\sqrt{1+\frac{t^2}{k^2}\cos^2(t(x\cdot v)) \sum_{i=1}^{n-1}\left(\tau_i(x)\cdot v\right)^2}\,\d\H^{n-1}(x)\\
&\hspace{0.3cm}\geq\liminf_{t\rightarrow\infty}
    \int_{G\setminus Z_t} \sqrt{1+\frac{t^2}{k^2}\delta^2\varepsilon^2}\,\d\H^{n-1}(x)\\
&\hspace{0.3cm}\geq \liminf_{t\rightarrow\infty}\frac{\lambda}{2}\sqrt{1+\frac{t^2}{k^2}\delta^2\varepsilon^2}=+\infty\,.
\end{align*}
Moreover, it holds that
\begin{equation}\label{eq:esttk}
t_k\leq Ck\,,
\end{equation}
where $C:=\sqrt{4r^2(\H^{n-1}(M))^2-\lambda^2}/(\lambda\delta\varepsilon)$. Let $\nu(x)$ be a unit normal vector to $M$ at $x$, for every $k\geq1$ let
\[
z_k(s):=\frac{1}{k}\sin(t_k s)\,,
\]
and define $w_k:M\rightarrow\R^n$ as
\[
w_k(x):=x+v_k(x)\nu(x)\,,
\]
where $v_k(x):=z_k(x\cdot v)\varphi_k(x)$.
Set $N_k:=w_k(M)$. Using the area formula (see \ref{areaformula}) we get
\[
\H^{n-1}(N_k)=\int_M J^M w_k\,\d\H^{n-1}
    =\int_M \sqrt{\mathrm{det}\left(\, \left[\nabla^M w_k\right]^T\cdot \nabla^M w_k \,\right)}\,\d\H^{n-1}\,.
\]
Since the above determinant is invariant under rotations, for every fixed $x\in M$ we can compute $\nabla^M w_k$ with respect to the orthonormal base of $\R^n$ given by $\tau_1(x),\dots,\tau_{n-1}(x)$,$\nu(x)$. It holds that
\begin{equation}\label{eq:nablaMwk}
\nabla^M w_k=\ov{\mathrm{Id}} + \nu\otimes \nabla^M(\varphi_k v_k) + (\varphi_k v_k) \nabla^M\nu
\end{equation}
where $\ov{\mathrm{Id}}$ denotes the $n\times(n-1)$ matrix defined as $(\ov{\mathrm{Id}})_{ij}:=\delta_{ij}$ for $i=1,\dots,n$ and $j=1,\dots,n-1$.
In particular, $\ov{\mathrm{Id}}^T$ is the projection over the first $n-1$ coordinates. Thus
\begin{equation}\label{eq:canc1}
\ov{\mathrm{Id}}^T\cdot (\nu\otimes \nabla^M(\varphi_k v_k))
    =(\nabla^M(\varphi_k v_k)\otimes \nu)^T\cdot\ov{\mathrm{Id}}=0.
\end{equation}
Moreover, since $\nabla^M\nu[\nu]=0$, we get
\begin{equation}\label{eq:canc2}
(\nabla^M(\varphi_k v_k)\otimes \nu)^T\cdot\nabla^M\nu
    =(\nabla^M\nu)^T\cdot(\nu\otimes(\nabla^M(\varphi_k v_k)))=0.
\end{equation}
Thus, from \eqref{eq:nablaMwk}, \eqref{eq:canc1} and \eqref{eq:canc2} and the identity $(a\otimes b)^T\cdot(a\otimes b)=|a|^2 b\otimes b$, we get
\begin{align*}
\left[\nabla^M w_k\right]^T\cdot  \nabla^M w_k &=\mathrm{Id}_{n-1}+\nabla^M(\varphi_k v_k)\otimes\nabla^M(\varphi_k v_k) +\varphi_k v_k \bigg[\, \ov{\mathrm{Id}}^T\cdot\nabla^M\nu\\
&\hspace{2.8cm}
    + (\nabla^M\nu)^T\cdot \ov{\mathrm{Id}}
    +\varphi_k v_k  (\nabla^M\nu)^T\cdot (\nabla^M\nu)\,\bigg].
\end{align*}
Using \eqref{eq:varphi} and \eqref{eq:esttk} it is possible to write
\[
\left[\nabla^M w_k\right]^T\cdot \nabla^M w_k = \mathrm{Id}_{n-1}+\nabla^M(\varphi_k v_k)\otimes\nabla^M(\varphi_k v_k)
    +(\varphi_kv_k)A_k\,,
\]
where the $A_k$'s are uniformly bounded. Using $\mathrm{det}(\mathrm{Id}+a\otimes a)=1+|a|^2$, we can write
\[
\mathrm{det}\left[\,\mathrm{Id}_{n-1}+\nabla^M(\varphi_k v_k)\otimes\nabla^M(\varphi_k v_k) \,\right]
    =1+|\nabla^M(\varphi_k v_k)|^2\,.
\]
Then
\begin{align}\label{eq:easydet}
\left|\, \int_M\sqrt{\mathrm{det}\left(\, \left[\nabla^M w_k\right]^*\cdot \nabla^M w_k \,\right)}\,\d\H^{n-1}
    -\int_M\sqrt{1+|\nabla^M(\varphi_k v_k)|^2}
        \d\H^{n-1}\,\right| \rightarrow 0 
\end{align}
since $A_k$ is uniformly bounded and $|\varphi_k v_k|\rightarrow0$ (by the uniform continuity of the determinant and a Taylor expansion).
Moreover, the fact that $\varphi^2_k|\nabla^M v_k|^2$ and $|v_k|^2|\nabla^M\varphi_k|^2$ are uniformly bounded, allows us to estimate
\begin{align}\label{eq:onlyck}
&\int_{M\setminus C_k} \sqrt{1+|\nabla^M(\varphi_k v_k)|^2}\,\d\H^{n-1} \nonumber \\
&\hspace{2cm}\leq \int_{M\setminus C_k} \sqrt{2|\nabla^M\varphi_k\cdot\nabla^M v_k|}\,\d\H^{n-1}\nonumber\\
&\hspace{3cm}+\int_{M\setminus C_k} \sqrt{1+\varphi^2_k|\nabla^M v_k|^2+|v_k|^2|\nabla^M\varphi_k|^2}\,\d\H^{n-1} \nonumber\\
&\hspace{2cm} \leq C\H^{n-1}(M\setminus C_k) + C\int_{M\setminus C_k} \sqrt{|\nabla^M\varphi_k|} \,\d\H^{n-1}\nonumber\\
&\hspace{2cm} \leq C(1+\sqrt{k})\H^{n-1}(M\setminus C_k) = \frac{C(1+\sqrt{k})}{k}\rightarrow0\,,
\end{align}
as $k\rightarrow\infty$. Thus, the combination of \eqref{eq:easydet} and \eqref{eq:onlyck} yields
\begin{equation*}
\left|\, \int_M\sqrt{\mathrm{det}\left(\, \left[\nabla^M w_k\right]^*\cdot \nabla^M w_k \,\right)}\,\d\H^{n-1}
    -\int_{C_k}\sqrt{1+|\nabla^M(\varphi_k v_k)|^2}\,\d\H^{n-1}\right|\rightarrow0\,,
\end{equation*}
as $k\rightarrow\infty$. Now, notice that for points in $C_k$ it holds
\[
1+|\nabla^M(\varphi_k v_k)|^2 = 1+|\nabla^M v_k|^2
    =1+\frac{t_k^2}{k^2}\cos^2(t_k(x\cdot v)) \sum_{i=1}^{n-1}\left(\tau_i(x)\cdot v\right)^2
\]
and thus by \eqref{eq:rper} we have that
\begin{equation}\label{eq:identityper}
\int_{C_k}\sqrt{1+|\nabla^M v_k |^2}\,\d\H^{n-1}=r\H^{n-1}(M)\,.
\end{equation}
Hence, by \eqref{eq:conv} and \eqref{eq:identityper}, we conclude that $\H^{n-1}(N_k)\rightarrow r\H^{n-1}(M)$, as $k\rightarrow\infty$. Finally, since $\varphi$ is compactly supported in $M$, $\pa M = \pa N_k$ for all $k \in \N^*$.
\end{proof}

We now combine the above results to obtain recovery sequences for absolutely continuous couples (see Remark \ref{rem:ac}).

\begin{proposition}\label{propo ac}
Let $(E,u)\in \mathfrak S$ be an absolutely continuous couple. Then, for every $\varepsilon>0$ there exists an absolutely continuous couple $(F,v)\in\mathfrak S$ such that
\[
\mathrm{d}_{\mathfrak S}[(F,v), (E,u)]<\varepsilon\,,\quad\quad\quad
|\F(F,v)-\ov{\F}(E,u)|<\varepsilon\,,
\]
and
\[
\left|\, \int_{\pared F} v\,\H^{n-1}-\int_{\pared E} u\,\H^{n-1}   \,\right|<\varepsilon\,.
\]

\end{proposition}

\begin{proof}
In the case $\psi=\ov{\psi}$, there is nothing to prove. Therefore, assume that there exists $s_0>0$ such that
$\psi\equiv\ov{\psi}$ in $[0,s_0]$ and $\ov{\psi}<\psi$ in $(s_0,\infty)$ (see Remark \ref{recipe}).
Let $(E_k,u_k)\in \mathfrak S$ and $ M^k_i \subset\partial E_k$ be the sequences given by Proposition \ref{prop:density} relative to $(E,u)$. Notice that, by looking at the way the $M^k_i$ are obtained, we can assume that each one of them is contained in a cube of diagonal $1/2k$ and of center $x_i^k$. Write
\[
u_k(x) =: \sum_{i=1}^\infty u_i^k \   \ca_{M_i^k}(x)\,.
\]
Using \eqref{eq:vicu}, and the extraction of a (not relabeled) subsequences, we can assume that
\begin{equation}\label{eq:vicuk}
\|u_k\|_{L^1(\partial E_k)}\leq \|u\|_{L^1(\partial^* E)}+\frac{1}{k}\,.
\end{equation}
Fix $k\in\N$ large enough and let
\begin{equation}\label{eq:erre}
r^k_i:=\max \left\{ 1,\ \frac{u_i^k}{s_0} \right\}\,.
\end{equation}
Let $\delta_k>0$ be such that $(\pa E_k)_{\delta_k}$ is a normal tubular neighborhood of the whole $\pa E_k$ to avoid self-intersection when wriggling.
By Lemma \ref{lem:wrigg} for every $i \in \N$ it is possible to find a sequence of smooth manifolds $(N_i^k)_{k\in \N}$ with Lipschitz boundary such that
\begin{equation}\label{eq:closeness}
N^k_i\subset (M^k_i)_{\varepsilon_i^k}\,, \qquad \ \
\left|\, \H^{n-1}(N^k_i)-r^k_i\H^{n-1}(M^k_i) \right|\leq\frac{2^{-i}}{k}\,,
\end{equation}
where $\varepsilon_i^k := \min( \delta_k, \frac{2^{-i}}{k})$.
Define
\begin{equation}\label{eq:vki}
v_i^k = \min\left\{ s_0, u_i^k \right\}\,.
\end{equation}
Observe that when $r_i^k=1$ then $N_i^k = M_i^k$ and $v_i^k = u_i^k$, \emph{i.e.}, we do not modify anything.

Now, let $F_k$ be the bounded set whose boundary is $\pa F_k := \bigcup_{i\in\N} \ov{N}^k_i$, and let
$v_k\in L^1(\partial F_k,\R_+)$ be defined as $v_k:=v^k_i$ on $N^k_i$.
Notice that $F_k$ is well defined, since the $N^k_i$ are disjoint, smooth and $\partial N^k_i=\partial M^k_i$ by construction. Then,
\[
\left\|\ca_{E_k}-\ca_{F_k}\right\|_{L^1}\leq\frac{1}{k}\,.
\]
Let $\varphi\in \mathcal C_c(\R^n)$. By uniform continuity of $\varphi$, fixed $\eta>0$ it is possible to find
$\bar{k}\in\N$ such that $|\varphi(x)-\varphi(y)|<\eta$ for every $x,y\in\R^n$ with $|x-y|<1/\bar{k}$. Increasing $k$ if necessary, we can assume that $1/k < 1/\ov{k}$. Then 
\begin{align*}
&\left|\, \int_{\partial F_k} \varphi v_k\,\d\H^{n-1} -\int_{\partial^* E} \varphi u\,\d\H^{n-1} \,\right| \\
&\hspace{2cm}=\left|\, \sum_{i\in\N} \int_{N^k_i} \varphi v_k\,\d\H^{n-1} - \int_{\partial^* E} \varphi u\,\d\H^{n-1} \,\right|\\
&\hspace{2cm}\leq \sum_{i\in \N} \left|\, \int_{N^k_i} \varphi v_k\,\d\H^{n-1} -
    \int_{M^k_i} \varphi u_k\,\d\H^{n-1} \,\right| \\
&\hspace{3cm}+\left|\, \int_{\partial E_k} \varphi u_k\,\d\H^{n-1}-\int_{\partial^* E} \varphi u\,\d\H^{n-1} \,\right|\\
&\hspace{2cm}\leq \eta \left(\, \|u_k\|_{L^1(\partial E_k)} + \|v_k\|_{L^1(\partial F_k)}  \,\right) \\
&\hspace{3cm}+\sup|\varphi|\, \sum_{i\in \N} \left|\,\H^{n-1}(N^k_i)v^k_i - \H^{n-1}(M^k_i)u^k_i \,\right|\\
&\hspace{4cm}+\left|\, \int_{\partial E_k} \varphi u_k\,\d\H^{n-1}-\int_{\partial^* E} \varphi u\,\d\H^{n-1} \,\right|.
\end{align*}    
In this last step we used the uniform continuity of $\varphi$, the facts that $M_i^k$ and $N_i^k$ are contained in cubes of diagonal $1/(2k)$ and $1/k$, respectively, and that $1/k < 1/\ov{k}$. Observe that the summands in the last term are zero if $r_i^k=1$, so denote $J \subset \mathbb{N}$ the set of indexes $i$ for which $r_i^k > 1$. We thus have
\begin{align*}    
&\left|\, \int_{\partial F_k} \varphi v_k\,\d\H^{n-1} -
    \int_{\partial^* E} \varphi u\,\d\H^{n-1} \,\right| \\
&\hspace{2cm}\leq\eta \left(\, \|u_k\|_{L^1(\partial E_k)} + \|v_k\|_{L^1(\partial F_k)}  \,\right)\\
&\hspace{3cm}+s_0 \sup|\varphi|\, \sum_{i\in J} \left|\,\H^{n-1}(N^k_i) - r_i^k \H^{n-1}(M^k_i) \,\right|\\
&\hspace{4cm}+\left|\, \int_{\partial E_k} \varphi 
    u_k\,\d\H^{n-1}-\int_{\partial^* E} \varphi u\,\d\H^{n-1} \,\right|\\
&\hspace{2cm}\leq 2\eta\left(\, \|u\|_{L^1(\partial^* E)} + \frac{1}{k} \,\right)
    +\frac{s_0 \sup|\varphi|}{k} \\
&\hspace{3cm}+\left|\, \int_{\partial E_k} \varphi 
    u_k\,\d\H^{n-1}-\int_{\partial^* E} \varphi u\,\d\H^{n-1} \,\right|
\end{align*}
where in the last step we used \eqref{eq:vicuk}, \eqref{eq:erre}, \eqref{eq:closeness} and \eqref{eq:vki}.
Now, by recalling that
\[
u_k\H^{n-1}\restr\partial E_k\wt u\H^{n-1}\restr\partial^* E\,,
\]
and using the arbitrariness of $\eta$, we conclude that the above quantities go to zero as $k\rightarrow\infty$. In particular  $\mu_k\wt\mu$,
where $\mu_k:=v_k\H^{n-1}\restr\partial F_k$ and $\mu:=u\H^{n-1}\restr\partial E$.
Moreover, $\mu_k(\R^n)\rightarrow\mu(\R^n)$.
Finally, observe that 
\begin{align*}
\left|\F(F_k,v_k) - \ov{\F}(E,u) \right|  \leq  & \left| \int_{\pa F_k} \psi(v_k) \d \H^{n-1} - \int_{\pa E_k} \ov{\psi}(u_k) \d \H^{n-1} \right| \\
&\hspace{1cm}+ \left|\int_{\pa E_k} \ov{\psi}(u_k) \d \H^{n-1} - \int_{\pared E} \ov{\psi}(u) \d \H^{n-1} \right|
\end{align*}
goes to zero as $k\to \infty$ thanks to similar computations of the ones above and (iv) of Proposition \ref{prop:density}. This concludes the proof.
\end{proof}

We now prove the approximation in energy of a measure $\mu$ that is singular with respect to $|D\ca_E|$.

\begin{proposition} \label{propo singular}
Let $\mu\in \mathcal{M}^+_{loc}(\R^n)$ be such that $\mu(\R^n) < \infty$. Then for every $\e>0$ there exists an absolutely continuous couple $(E,u)$ such that 
	\begin{align*}
	d_{\mathfrak S}[(E,u) , (\varnothing, \mu) ]<\e\,,\quad\quad\quad \left| \F(E,u) -  \Theta \mu(\R^n) \right|< \e\,,
	\end{align*}
and
\[
\left|\, \int_{\pared E} u\,\H^{n-1}-\mu(\R^n)\,\right|<\varepsilon\,.
\]
\end{proposition}

The proof of Proposition \ref{propo singular} is a consequence of the following lemma.

\begin{lemma} \label{technical lemma singular}
Let $f\in \C^{\infty}(\R^n)\cap L^1(\R^n)$ with $f\geq 0$. For every $\e>0$ there exists an absolutely continuous couple $(F,w)$ such that
	\begin{align*}
	d_{\mathfrak S} [(F,w), (\varnothing,f\L^n)] &< \e \quad\quad\quad\left| \ov{\F} (F,w) - \ov{\F}(\varnothing,f\L^n) \right| < \e\,,
	\end{align*}
and
\[
\left|\, \int_{\pared F} w\,\H^{n-1}-\int_{\R^n} f\,\mathrm{d}x  \,\right|<\varepsilon\,.
\]

\end{lemma}

Before proving this lemma, we first show how to derive Proposition \ref{propo singular} from it.
 
\begin{proof}[Proof of Proposition \ref{propo singular}]
Let $\{\eta_{r}\}_{r>0}$ be a mollifying kernel, and define	
\[
f_r(x):=\int_{{\R^n}} \eta_{r}(x-y) \d \mu(y).
\]
By standard arguments we know that $f_r\in C^{\infty}(\R^n)$ and $f_r\L^n \wt \mu$ as $r\rightarrow 0$. In particular, for every $\e>0$ we can find $\de>0$ such that
	\begin{align*}
	d_{\M}(f_{\delta} \L^n, \mu) & < \e/3\,,
	\end{align*}
and
\[
\left|\, \int_{\R^n} f_\delta\,\mathrm{d}x -\mu(\R^ n) \,\right|<\varepsilon/3\,. 
\]
Moreover, since
	\[
	\int_{\R^n} f_r \d x \underset{r\to 0}{\longrightarrow} \mu(\R^n),
	\]
up to further decreasing $\delta$ we can also ensure that
	\[
	\left| \ov{\F} (\varnothing, f_{\delta} \L^n) - \ov{\F}(\varnothing, \mu )\right| = \Theta \left| \ \|f_\de\|_{L^1} - \mu(\R^n)   \right| < \e/3.
	\]
Applying Lemma \ref{technical lemma singular} we find an absolutely continuous couple $(F ,w ) $ such that
	\begin{align*}
	d_{\mathfrak S}[ (F,w) , (\varnothing,f_{\delta} \L^n) ] < \e/3, \qquad \quad \left| \ov{\F}(F,w) -  \ov{\F} (\varnothing, f_{\delta} \L^n) \right| < \e/3\,,
	\end{align*}
and
\[
\left|\, \int_{\pared F} w\,\H^{n-1}-\int_{\R^n} f_\delta\,\mathrm{d}x  \,\right|<\varepsilon/3\,.
\]
Applying Proposition \ref{propo ac} let $(E,u) $ be an absolutely continuous couple such that 
	\begin{align*}
	d_{\mathfrak S}[ (E,u),(F,w) ] < \e/3, \qquad \quad \left| \F(E,u) -  \ov{\F} (F,w) \right| < \e/3.
	\end{align*}
Using the triangle inequality, we conclude that
		\begin{align*}
	d_{\mathfrak S}[ (E,u),(\varnothing, \mu) ] <\e, \qquad \quad 	\left| \F(E,u) -  \ov{\F} (\varnothing, \mu ) \right| < \e\,,
	\end{align*}
as well as
\[
\left|\, \int_{\pared E} u\,\H^{n-1}-\mu(\R^n)\,\right|<\varepsilon\,.
\]
\end{proof}

\begin{proof}[Proof of Lemma \ref{technical lemma singular}]
Let us first notice that, without loss of generality, it is possible to assume that $f\in C^\infty_c(\R^n)\cap L^1(\R^n)$.
Indeed, it is possible to find $R>0$ such that
\[
\int_{\R^n\setminus B_R(0)} |f(x)| \d x<\e/2.
\]
Taking a mollifying kernel $\{\eta_r\}_{r>0}$ it is possible to find $r>0$ such that the function
\[
\widetilde{f}_{r}(x):=\int_{B_R(0)} \eta_r(x-y) f(y) \d y
\]
satisfies $\|f-\widetilde{f}_r\|_{L^1(\R^n)}<\e$. In particular we will assume $f\in \C_c^{\infty}(\R^n)$ (by performing the slight abuse of notation $f=\widetilde{f}_r$).\\

Let $\{Q_j^{k}\}_{j\in \N}$ be a diadic partition of $\R^n$  in cubes of size $|Q_j^{k}|= 2^{-nk}$ and centers $x_j^k$. We introduce the set of indexes
	\[
	 J_0 =\{ j \in \{ 1,\ldots, 2^{nk} \}  :  \ |Q_j^k\cap \ov{\{f>0\}}|\neq 0 \}\,,
	\]
and we set
	\[
	0<m_k:= \min \left\{  \int_{Q_j^k} f \d x \ : \  j\in J_0 \right\}< \sup_{\R^n}\{f\} 2^{-nk}.
	\]
Since $\text{supp}(f)$ is compact, we can infer that
	\begin{equation}\label{controlpower}
	\#(J_0)|Q_j^k|<C
	\end{equation}
where here, and in what follows, $C$ will always stand for a constant depending on $f$ and $n$ only and whose value can change from line to line.  Let 
	\[
	r_k:=  m_k^{1/(n-1)}2^{-2k},  \ \ \ \ \ B_j^k:= B_{r_k}(x_j^k)\cc Q_j^k\,,
	\]
and define (see Figure \ref{pic: sing})
	\begin{equation}\label{eqn: sequence of mickey mouse}
			F_k:= \bigcup_{j\in J_0 }  B_j^k, \ \ \ \ \ 	w_k(x):=\sum_{j\in J_0} \frac{ \ca_{\pa B_j^k}(x) }{P(B_j^k)} \int_{Q_j^k} f(y) \d y.
	\end{equation}
	
\begin{figure}[h!]
\centering
\includegraphics[scale=0.6]{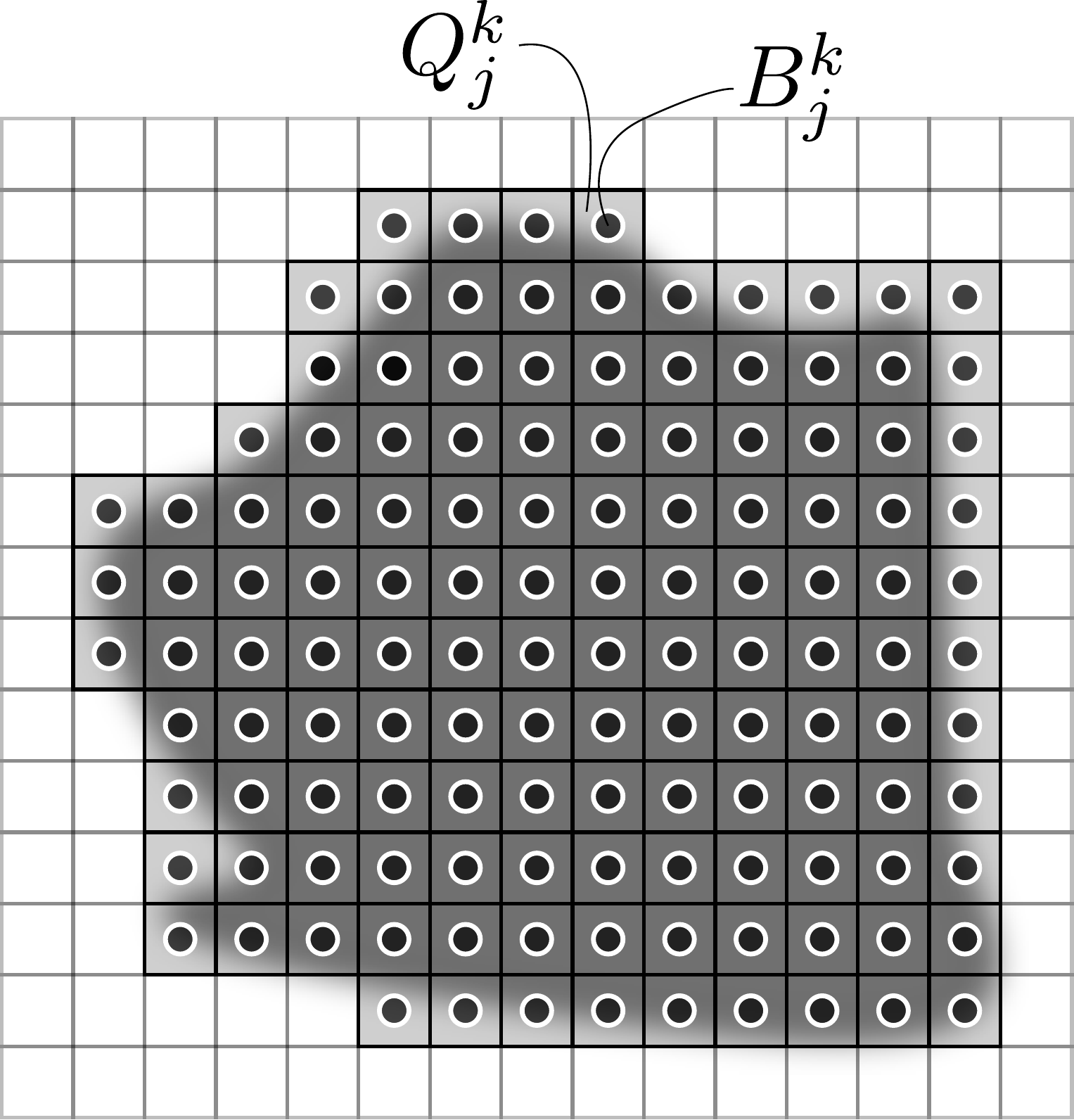}\caption{In the background the set $\text{supp} (f)$. On the top the diadic division and the set $F_k$ built as the union of small balls (in black). The adatom density $w_k$ is defined to be constant on each $\pa B_j^k$ (evidenced in white circles).}\label{pic: sing}
\end{figure}

Notice that, since $B_j^k\cap B_m^k=\varnothing$ for $j\neq m$, the function $w_k \in L^1(\pared F_k;\R_+)$ is well defined.
We also notice that, by construction, for each $j\in J_0$ it holds
	\begin{equation}\label{pieceofcake}
	\frac{1}{P(B_j^k)} \int_{Q_j^k} f(y)\d y \geq C 2^{2(n-1)k}.
	\end{equation}
Since $\ov{\psi}(x)/x \searrow \Theta$ we have, for each $\e>0$ and for $k$ big enough, that
	\begin{equation}\label{pieceofcake2}
	 \left|P(B_j^k)\ov{\psi}\left( \frac{1}{P(B_j^k)} \int_{Q_j^k} f\d y\right)-\Theta \int_{Q_j^k} f\d y \right|<\e \int_{Q_j^k} f\d y,
	\end{equation}
for all $j\in J_0$. Since
	\begin{align*}
	\ov{\F}(F_k,w_k)&= \sum_{j\in J_0} \int_{\pa B_j^k}\ov{\psi}(w_k) \d \H^{n-1}\\
	&=    \sum_{j\in J_0} P(B_j^k) \ov{\psi}\left( \frac{1}{P(B_j^k)} \int_{Q_j^k} f\d y\right)\\
	&= \sum_{j\in J_0} \left( P(B_j^k) \ov{\psi}\left( \frac{1}{P(B_j^k)} \int_{Q_j^k} f\d y\right)
	    -\Theta \int_{Q_j^k} f\d y\right)\\
	&\hspace{2cm}+\Theta \sum_{j\in J_0}  \int_{Q_j^k} f\d y,
		\end{align*}
invoking \eqref{pieceofcake2} and \eqref{controlpower}, for large $k$, we are led to
	\begin{equation*}
		\left|\ov{\F}(F_k,w_k) - \Theta\int_{\R^n} f\d y\right| \leq \e \sum_{j\in J_0} \int_{Q_j^k} f\d y\leq \e C.
	\end{equation*}
	
We now claim that the sequence $((F_k,w_k))_{k\in\N}$ defined in \eqref{eqn: sequence of mickey mouse} converges to $(\varnothing,f \L^n)$. Using \eqref{controlpower} together with the definition of the $r_k$'s, we get that $|F_k| \rightarrow 0$, and thus $\ca_{F_k} \rightarrow 0$ in $L^1$. Let $\mu_k := w_k \H^{n-1}\restr \pa F_k$ and $\mu := f\L^n$. Noticing that
	\begin{equation*} 
	\mu_k(\R^n)=\mu(\R^n)<+\infty\,,
	\end{equation*}
by Lemma \ref{weak*comp}, up to a (not relabeled) subsequence,
we have that $\mu_k \wt \nu$ for some $\nu\in\mathcal{M}^+_{loc}(\R^n)$. In order to prove that $\nu = f \L^n$, we compute its density. For this, for any ball $B_r $  we introduce the subset of indexes
	\begin{align*}
	\text{in}(B_r;k)&:= \{j \in \{ 1,\ldots, 2^{nk} \} \ : \ Q_j^k\cc B_r)\}, \\
	 \text{bd}(B_r;k)&:=\{ j \in \{ 1,\ldots, 2^{nk} \} \ : \ Q_j^k\cap \pa B_r\neq \varnothing\}.
	\end{align*}

\noindent\textbf{Step one:} \textit{estimate on the cardinality of $\bd(B_r;k)$: $\# (\bd(B_r;k)) $.} Notice that if $Q_j^k\cap \pa B_r \neq \varnothing$ then
	\[
	Q_j^k \subseteq \{x\in \R^n \ : \ d(x,\pa B_r) \leq \sqrt{n} 2^{-k} \} 
	\]
since $\sqrt{n} 2^{-k}$ is the diagonal of each cube. Observe that
	\[
	\left|\{x\in \R^n \ : \ d(x,\pa B_r) \leq \sqrt{n}  2^{-k} \} \right| \leq C(n)P(B_r) 2^{-k},
	\]
and thus we have
	\begin{equation}\label{eqn: crucial estimate}
	\# (\text{bd}(B_r;k))  \leq CP(B_r)  2^{(n-1)k}.
	\end{equation}
\text{}\\
\textbf{Step two:} \emph{$\nu = f\L^n$.} Let $x\in \text{supp}(f)$, $r>0$, $B_r=B_r(x)$, and consider 
	\[
	D_r(k):= \bigcup_{j\in \text{in}(B_r;k)} Q_j^k.
	\]
In view of \eqref{eqn: crucial estimate}, we have
	\begin{align}\label{eqn: pavarotti}
	|B_r \setminus D_r(k)| &\leq  \H^{0} (\text{bd}(B_r;k) )  |Q_j^k| \leq  C P(B_r) 2^{-k}\,. 
	\end{align}
Notice also that 
	\begin{equation}\label{eqn: Napoleone}
	\mu_k(D_r(k))= \sum_{j\in \text{in}(B_r ;k)} \mu_k(Q_j^k) = \sum_{j\in \text{in}(B_r;k)} \mu(Q_j^k) =  \int_{D_r(k)} f \d x.
	\end{equation}
Thus \eqref{eqn: Napoleone} and \eqref{eqn: pavarotti} imply that
	\begin{equation}\label{eqn: part 1}
	\left|\mu_k(D_r(k)) - \int_{B_r}f \d x \right| \underset{k\to \infty}{\longrightarrow} 0.
	\end{equation}
Also, by \eqref{eqn: crucial estimate}, we have
	\begin{align*}
	\left| \mu_{k}(B_r) -\mu_k(D_r(k)) \right| &\leq \sum_{ j\in   \text{bd}(B_r;k) } \int_{Q_j^k} f\d y  \leq C  \#(\text{bd}(B_r;k)) 2^{-nk}\\
	&\leq C  P(B_r)  2^{-k}\underset{k\to \infty}{\longrightarrow} 0.
	\end{align*}
By the triangle inequality and \eqref{eqn: part 1} we obtain
	\begin{equation}\label{eqn: damn it...}
	\left| \mu_k(B_r) - \int_{B_r} f\d x  \right| \leq \left| \mu_k(B_r) -  \mu_k(D_r(k)) \right|+ \left| \mu(D_r(k)) -\int_{B_r} f\d x  \right| \rightarrow 0.
\end{equation}
Clearly, if $x\notin \text{supp}(f)$ we have $\mu_k(B_r(x))=0$ for a small enough $r>0$ and for a large enough $k$, implying that $\nu(B_r(x))=0$. On the other hand, in view of \eqref{eqn: damn it...}, if $x\in \text{supp}(f)$ then for every $r>0$
	\[
	\mu_{k_h}(B_r(x)) \rightarrow \int_{B_r(x)} f \d y\,. 
	\]
Thus, by \eqref{foliation} for all but countably many $r>0$
	\[
	\mu_{k_h}(B_r(x))\rightarrow \nu(B_r(x)).
	\]
This argument shows that
	\begin{equation}		
		\lim_{r\rightarrow 0} \frac{\nu(B_r(x))}{r^n} = \left\{ 
		\begin{array}{ll}
		0 & \text{if $x\notin \text{supp}(f)$},\\
		\om_n f(x) & \text{if $x\in \text{supp}(f)$}\,,
			\end{array}
		\right.
	\end{equation}
and hence $\nu=f\L^n$. Since the limit measure $\nu$ does not depend on the subsequence $\mu_{k_h}$, we conclude that $\mu_k \wt f\L^n$.
\end{proof}

We are finally in position to prove the relaxation result.

\begin{proof}[Proof of Theorem \ref{thm:relax}]
\textbf{Step one:} \textit{liminf inequality.} Let $(E,\mu)\in \mathfrak S$ and let $\left((E_k,\mu_k)\right)_{k\in\N}\subset \mathfrak S$ with
$(E_k,\mu_k)\rightarrow (E,\mu)$.
If there exists $\bar{k}\in\mathbb{N}$ such that $\mu_k$ has a singular part
with respect to $|D\ca_{E_k}|$ for all $k\geq\bar{k}$, then $\F(E_k,\mu_k)=\infty$ for all $k\geq\bar{k}$.
So we can assume, without loss of generality, that, up to a (not relabeled) subsequence, $\mu_k=u_k|D\ca_{E_k}|$, with
$u_k\in L^1(\partial^* E_k,\R_+)$ for all $k\in\N$.
Since $\ov{\psi}\leq\psi$, we have that
\[
\F(E_k,\mu_k)=\int_{\partial^* E_k} \psi(u_k)\,\d\H^{n-1} \geq \int_{\partial^* E_k} \ov{\psi}(u_k)\,\d\H^{n-1}
    =\ov{\F}(E_k,\mu_k)\,.
\]
Using the semi-continuity of $\ov{\F}$ (see Lemma \ref{thm:lsc}), we get that
\[
\liminf_{k\rightarrow\infty}\F(E_k,\mu_k)\geq \liminf_{k\rightarrow\infty}\ov{\F}(E_k,\mu_k)\geq\ov{\F}(E,\mu)\,.\\
\]

\textbf{Step two:} \textit{limsup inequality.} Let $(E,\mu)\in \mathfrak S$ and write $\mu = u|D\ca_E| + \mu^s$, where $\mu^s$ is the singular part of $\mu$ with respect to $|D\ca_E|$. Set $\widetilde{m}:=|E|+\mu(\R^n)$. The cases $\widetilde{m}\in\{0,\infty\}$ are trivial, so we can assume $\widetilde{m}\in(0,\infty)$.
For every $k\in\N^*$, using Propositions \ref{propo ac} and \ref{propo singular}, we can find
$(F_k, v_k)$ and $(G_k,w_k)$ in $\mathfrak S$ such that 
	\begin{align}
	\label{eps31} d_{\mathfrak S}\left[ (E,u), \  (F_k,v_k) \right] &< 1/(4k)\,, \\ 
	\label{eps32} d_{\mathfrak S}[ (G_k,w_k),\  (\varnothing,\mu^s)] &< 1/(4k)\,, \\
	\label{eps21}\left| \int_{\partial^* F_k} \psi(v_k) \d\H^{n-1} 
	    - \int_{\partial^* E} \ov{\psi}(u) \d\H^{n-1} \right| &< 1/(2k)\,,  \\
	\label{eps22}\left| \int_{\partial^* G_k} \psi(w_k) \d\H^{n-1} - \Theta \mu^s(\R^n) \right| &< 1/(2k)\,,\\
	\label{epsm1}\left|\, \int_{\pared F_k} v_k\,\H^{n-1}-\int_{\pared E}u\,\H^{n-1} \,\right|&<1/(2k)\,,\\
	\label{epsm2}\left|\, \int_{\pared G_k} w_k\,\H^{n-1}-\mu^s(\R^n) \,\right|&<1/(2k)\,.
	\end{align}
Define $\widetilde{E}_k:=F_k\triangle G_k$, the symmetric difference of $F_k$ and $G_k$. Up to arbitrarily small isometries of the (finitely many) connected components of $G_k$, it is possible to assume that (see \cite{Mattila2})
$$\H^{n-1}( \partial^* F_k\cap \ov{G_k} )= 0\,,$$
and that \eqref{eps32} still holds. In particular
	\[
	\H^{n-1}(\partial^* \widetilde{E}_k)=\H^{n-1}(\partial^* F_k)+\H^{n-1}(\partial^* G_k).
	\]
Using $|\,|a|-|b|\,|\leq|a-b|$, we obtain
\begin{align}\label{eq:convsets}
\|\ca_E-\ca_{\widetilde{E}_k}\|_{L^1}&=\|\, \ca_E-|\ca_{F_k}-\ca_{G_k}|\,\|_{L^1} \nonumber\\
&\leq \|\ca_E-\ca_{F_k}\|_{L^1}+\|\ca_{G_k}\|_{L^1}\leq 1/(2k)\,.
\end{align}
Now, define $\widetilde{u}_k:\partial^* \widetilde{E}_k\rightarrow\R_+$ as
\[
\widetilde{u}_k(x):=
\left\{
\begin{array}{ll}
v_k(x) & \text{ if } x\in\partial^* F_k\,,\\
w_k(x) & \text{ if } x\in\partial^* G_k\,.
\end{array}
\right.
\]
Using \eqref{eq:convsets}, \eqref{epsm1} and \eqref{epsm2} we get the existence of $(\varepsilon_k)_{k\in\N}$ with $\varepsilon_k\rightarrow1$
such that
\[
|E_k|+\int_{\pared E_k} u_k\,\H^{n-1}=\widetilde{m}\,,
\]
where $E_k:=\varepsilon_k \widetilde{E}_k$ and $u_k:\pared E_k\rightarrow\R_+$ is defined as
$u_k(x):=\widetilde{u}_k(\varepsilon_k^{-1}x)$.
Moreover, up to a (not relabeled) subsequence, we can assume that \eqref{eps31}, \eqref{eps32}, \eqref{eps21} and
\eqref{eps22} still hold true.

Set $\mu_k:=u_k\,\H^{n-1}\restr\partial^* E_k$. Using \eqref{eps31} and \eqref{eps32}, we get that
$$d_{\mathcal M} (\mu,\ \mu_k) < 1/(2k)\,,$$
and with similar computations as in \eqref{eq:convsets}, we get $\|\ca_E-\ca_{E_k}\|_{L^1}\leq 1/(2k)$.
Thus
\begin{equation*}
d_{\mathfrak S}\left[ (E, \mu), \  (E_k, \mu_k) \right] < 1/k\,.
\end{equation*}
Finally, noticing that
\begin{align*}
\mathcal F(E_k,\mu_k) &= \varepsilon_k^n\int_{\partial^* E_k} \psi(u_k) \d\H^{n-1}\\
&=\varepsilon_k^n  \int_{\partial^* F_k} \psi(v_k) \d\H^{n-1}
    +\varepsilon_k^n\int_{\partial^* G_k} \psi(w_k) \d\H^{n-1}\,,
\end{align*}
and using \eqref{eps21}, \eqref{eps22} and $\varepsilon_k\rightarrow1$, we get
\begin{equation*}
\left|  \mathcal F(E_k,\mu_k) -  \overline{\mathcal F}(E, \mu) \right| < 1/k\,.
\end{equation*}
Thus, $\left((E_k,u_k)\right)_{k\in\N}$ is the desired recovery sequence.
\end{proof}

\begin{remark}\label{rem:convconstraint}
Notice that the above proof provides, for any $(E,\mu)\in\mathfrak S$ with $\mu(\R^n)<\infty$, a recovery sequence
$\left((E_k,u_k)\right)_{k\in\N}$ with
\[
|E_k|+\mu_k(\R^n)=|E|+\mu(\R^n)\,.
\]
\end{remark}


  \section{Minimizers and critical points of the relaxed energy}
  
  \label{secminrelax}
  
We now study minimizers and critical points of the relaxed energy $\ov{\F}$ and their relation with those of $\F$.

\begin{theorem}\label{critfbar}
Assume that $\psi$ is strictly convex. Let $(E,\mu) \in \mathfrak S$ be such that $|E|>0$ and its absolutely continuous part $(E,u)$ is a regular critical point for $\ov{\F}$, \emph{i.e.}, $(E,u)$ is as in Definition \ref{def:av} and satisfies
\begin{equation}\label{eq:critfbar}
\int_{\pa E} [\,\ov{\psi}\,'(u) w+\ov{\psi}(u) v H \,] \d \H^{n-1}=0 \quad\, \text{ for all } (v,w)\in \mathrm{Ad}(E,u)\,,
\end{equation}
where $\mathrm{Ad}(E,u)$ is defined in Definition \ref{def:av}.
Then $E$ is a ball $B$ with constant adatom density $c < s_0$ satisfying condition \eqref{cstcritpointsvalue}, namely
\[
(\psi(c)-c\psi'(c))H_{\pa B} = \rho \psi'(c)\,.
\]
\end{theorem}

\begin{proof}
Notice that $(E,u)\in\mathrm{Cl}(\widetilde{m})$, where $\widetilde{m}:=m-\mu^s(\mathbb R^n)$. Since $|E|>0$ we have that
$\widetilde{m}>0$.
In the case $\psi = \ov{\psi}$ the result follows using the same steps of the proof of Proposition \ref{prop:charregcrit} applied to the couple $(E,u)\in\mathrm{Cl}(\widetilde{m})$.

Otherwise, we will obtain the result by adapting the same proof as follows:
Step one implies that, on each connected component of $\pa E$, $\ov{\psi}\,'(u)$ is constant.
Thus, for every fixed connected component $(\partial E)_i$ of $\partial E$, we have two possibilities: $\ov{\psi}\,'(u) \equiv \Theta$ or $\ov{\psi}\,'(u) < \Theta$.

In the first case $u \geq s_0$ $\H^{n-1}$- a.e. on $(\partial E)_i$, so that $\ov{\psi} - u\ov{\psi}\,'(u) \equiv 0$.
We claim that this is impossible. Indeed, arguing as in Step two of Proposition \ref{prop:charregcrit},
take $v\in \mathcal{C}^1((\pa E)_i)$ such that
\begin{equation}\label{eq:nonvanishv}
\int_{(\pa E)_i}v\d\H^{n-1}\neq0\,,
\end{equation}
and consider the  admissible velocities $(v,-v(uH+\rho))\in \mathrm{Ad}(E,u)$.
Using the fact that $u$ is constant on $(\pa E)_i$ and \eqref{eq:critfbar}, we obtain
\begin{align*}
0 &=\left( \ov{\psi}(u)-u\Theta \right) \int_{(\pa E)_i} vH_{\pa E}\d\H^{n-1} - \rho\Theta\int_{(\pa E)_i}v\d\H^{n-1}\\
&=- \rho\Theta\int_{(\pa E)_i}v\d\H^{n-1}\neq0\,,
\end{align*}
where in the last step we used \eqref{eq:nonvanishv} and that $\rho, \Theta\neq0$.

So, we have that, on each connected component of $\partial E$, $\ov{\psi}\,'(u) < \Theta$, that in turn implies that 
$u< s_0$ $\H^{n-1}$-a.e. on $\partial E$. But for such values of $u$, the functions $\psi$ and $\ov{\psi}$ agree.
Thus we can conclude by arguing as in steps 2,3 and 4 of the proof of Proposition \ref{prop:charregcrit}.
\end{proof}

\begin{remark}
The necessary condition $c<s_0$ is physically relevant and it prevents, in the case $\psi\not\equiv\ov{\psi}$, the occurrence of large concentrations of atoms freely diffusing on the surface of the crystal. It will have a considerable importance in the study of gradient flows associated to $\ov{\F}$, as it will lead them to be attracted by points nearby which the equations are \textit{parabolic} (parabolicity will be given by $\psi(c) - c \psi'(c)>0$, \textit{i.e.}, by $c<s_0$).
\end{remark}

We now prove that the minimum of $\ov{\F}$ can be reached by balls with constant adatom density. Observe that due to the previous theorem, the density cannot be arbitrarily big (the balls cannot be arbitrarily small), even though a Dirac delta $(\varnothing, \delta)$ could still be a minimizer since this is not an absolutely continuous couple.

\begin{definition}
Fix $m>0$ and set
\[
\ov{\gamma}_m:=\inf\{\, \ov{\F}(E,\mu) \,:\, (E,\mu)\in \ov{\mathrm{Cl}}(m) \,\}\,,
\]
where
\[
\ov{\mathrm{Cl}}(m):=\left\{\, (E,\mu )\in \mathfrak S \,:\, \ov{\mathcal J}(E,\mu)=m \,\right\}\,,
\]
and
\[
\ov{\mathcal J}(E,\mu):=\rho|E| + \mu(\R^n)\,.
\]
\end{definition}

\begin{theorem}\label{rolexx}
Fix $m>0$. If $\psi$ satisfies the assumptions of Theorem \ref{thm:ex},
then there exist $R\in(\underline{R}_m,\ov{R}_m)$ and a constant $0 < c < s_0$ such that the pair $(B_R,c)\in\ov{\text{Cl}}(m)$, and
\[
\ov{\F}(B_R,c)=\ov{\gamma}_m = \gamma_m\,.
\]
Moreover, every minimizing couple $(E,\mu)\in\ov{\text{Cl}}(m)$ is such that either
$E$ is a ball or $E=\varnothing$.
\end{theorem}

\begin{proof}
Let $(E,\mu)\in\ov{\text{Cl}}(m)$
and let $((E_k,u_k))_{k\in\N}~\subset \mathfrak S$ be a recovery sequence given by Theorem \ref{thm:relax}, \emph{i.e.},
\[
\F(E_k,u_k)\rightarrow\ov{\F}(E,\mu)\,.
\]
By Remark \ref{rem:convconstraint} we have that
\begin{equation}\label{eq:convconstr}
\mathcal J(E_k,u_k)=\ov{\mathcal J}(E,\mu)\,.
\end{equation}
By Theorem \ref{thm:ex} we know that there exist $R\in(\underline{R}_m,\ov{R}_m)$ and $c>0$ such that
\[
\mathcal  J(B_R,c)=\mathcal  J(E_k,u_k)\,,\quad\quad
\F(B_R,c)=\gamma_{m}\,.
\]
Moreover, if $E_k$ is not a ball, then
\[
\F(B_R,c)<\F(E_k,u_k)\,.
\]
Thus
\[
\ov{\F}(E,\mu)=\lim_{k\rightarrow\infty}\F(E_k,u_k)\geq\F(B_R,c)=\ov{\F}(B_R,c)\,.
\]
In particular, if we take $((F_k,w_k))_{k\in\N}$ to be a minimizing sequence for the constrained minimization problem for $\ov{\F}$, we get that
\[
\ov{\gamma}_m=\lim_{k\rightarrow\infty}\ov{\F}(F_k,w_k)\geq\ov{\F}(B_R,c)\geq\ov{\gamma}_m\,,
\]
that is $\ov{\F}(B_R,c)=\ov{\gamma}_m$.

Finally, let $(E,u|D\ca_E|+\mu^s)$ be a minimizer of $\ov{\F}$ in $\ov{\mathrm{Cl}}(m)$ with $|E|>0$ and assume $E$ is not a ball. Set $m_1:=m-\mu_s(\R^N)>0$. Then $(E, u)\in\mathrm{Cl}(m_1)$. Thus, applying Lemma \ref{reductionto} to this couple, we get that
\[
\F(E,u)>\F(B,\ov{u})\,,
\]
where $B$ is a ball with $|B|=|E|$ and $\ov{u}:=\fint_{\pared E}u\,\d\H^{n-1}$.
Then $(E,u|D\ca_E|+\mu^s)\in\ov{\mathrm{Cl}}(m)$ and
\[
\ov{\F}(E,u|D\ca_E|+\mu^s)>\ov{\F}(B,\ov{u}|D\ca_B|+\mu^s)\,,
\]
which is in contradiction with the minimality of $(u|D\ca_E|+\mu^s)$.
\end{proof}

\begin{remark}
We would like to point out that the strategy we used to deal with this "constrained relaxation" problem is not usual. Indeed, it is more customary to insert the mass constraint in the definition of the functional, \emph{i.e.}, define for $m>0$,

\[
\F_m(E,\mu):=
\left\{
\begin{array}{ll}
\displaystyle\int_{\partial^*E} \psi(u)\,\d\H^{n-1} & \text{ if } \mu = u|D\ca_E| \text{ with } (E,u)\in \mathrm{Cl}(m)\,,\\
&\\
+\infty & \text{ otherwise}\,,
\end{array}
\right.
\]
and then compute the relaxation of $\F_m$. We avoided to do that because we were able to recover the energy of every $(E,\mu)\in\mathfrak{S}$ satisfying $\ov{\mathcal J}(E,\mu)=m$ with sequences satisfying the same mass constraint, as explained in Remark \ref{rem:convconstraint}.
\end{remark}

\begin{remark}
Minimizers of $\ov{\F}$ can have less structure than minimizers of $\F$ in the following terms:
\begin{enumerate}[i)]
\item the additivity of the singular part of $\ov{\F}$ allows for a huge variety of phenomena. For instance, if $\Theta \gamma_m = m$, any couple of Dirac deltas suitably weighted will produce a minimizing couple $(\varnothing, m_1 \delta_1 + m_2 \delta_2)$.
\item for the same reason, if there exists a minimizer $(E,u|D\ca_E|+\mu^s)$ with a non-zero singular part $\mu^s$, any couple $\mu_1^s, \mu_2^s$ such that $(\mu_1^s + \mu_2^s)(\R^n) = \mu^s(\R^n)$ will produce another minimizer $(E,u|D\ca_E| + \mu_1^s + \mu_2^s)$.
\end{enumerate}
Observe that there are two distinct ways of seeing a ball with constant adatom density in our setting.
One is $(B_{R(c)}, c)$ representing a ball of crystal with a constant adatom density on its surface. Another is
$(\varnothing, \rho \ca_{B_{R(c)}} \L^n + c \H^{n-1}\restr\pa B_{R(c)})$
These representations have the same mass but the former one is better energetically, provided
$$ \ov{\psi}(c) \leq \Theta c + \frac{\Theta  \rho R(c)}{n}.$$
\end{remark}


\appendix  

\section{Convex subadditive envelope of a function}\label{seccvxsubadd}

\begin{definition}
Let $g:\R\rightarrow\R$. We say that $g$ is \emph{subadditive} if for every $r,s\in\R$,
\[
g(r+s)\leq g(r)+g(s)\,.
\]
\end{definition}

\begin{definition}\label{def:convexsubadditiveenvelope}
Let $g:[0,\infty)\rightarrow\R$ be a function. We define its \emph{convex subadditive envelope}
$\text{convsub}(g):[0,\infty)\rightarrow\R$ as
\[
\text{convsub}(g)(s):=\sup\{\, f(s) \,:\, f:[0,\infty)\rightarrow\R \text{ is convex, subadditive and }f\leq g \,\}\,.
\]
\end{definition}

The aim of this section is to characterize the convex subadditive envelope of admissible energy densities (see  Definition \ref{def:energy}). To this end, we need a few preliminary results which are related to the parabolicity condition \eqref{paracond}.

\begin{lemma}\label{cvxsubaddparacd}
Let $g: (0,+\infty) \to \mathbb R$ be convex and subadditive. Then, $s\mapsto g(s)/s$ is non-increasing in $(0,+\infty)$.
In particular for $\mathcal{L}$-a.e. $s\in\R$ we have
\[
g(s) - g'(s)s \geq 0\,.
\]

\end{lemma}

\begin{proof}
Assume, by contradiction,  that there exist $0<r<s$ with
\begin{equation}\label{eq:contr}
\frac{g(r)}{r} < \frac{g(s)}{s}\,.
\end{equation}
Let $t:=r+s$. By subadditivity, we get 
\begin{equation}\label{eq:convineq}
\frac{g(t)-g(r)}{t-s}=\frac{g(r+s)-g(s)}{r} \leq \frac{g(r)}{r}\,.
\end{equation}
Moreover, \eqref{eq:contr} yields
$$\frac{g(r)}{r} < \frac{g(s)-g(r)}{s-r}.$$
These two inequalities together violate the convexity of $g$.

Finally, since $r \mapsto g(r)/r$ is non-increasing, it is differentiable $\mathcal{L}$-a.e. on $\R$. In particular, fixed $r\in\R$ for which $g'(r)$ exists, we have that
\[
g'(r)=\lim_{s\rightarrow r^+}\frac{g(s)-g(r)}{s-r}\leq\frac{g(r)}{r}\,,
\]
where in the last step we used \eqref{eq:convineq}.
\end{proof}

\begin{lemma}\label{noninc}
Let $g:(0,+\infty)\rightarrow\R$ be a convex function. Let $D\subset\R$ be the set where $g'$ is defined.
Then, the function $r \mapsto g(r)-g'(r)r$ is non-increasing on $D$.
\end{lemma}

\begin{proof}
It suffices to observe that for any $0 < r \leq  s$, since $g'$ is a.e. non-decreasing and $r< 0$,
$$g'(s)s - g'(r) r \geq g'(s)(s-r) \geq \int_r^s g'(t) \,\mathrm{d} t = g(s)-g(r)$$
\end{proof}

We now recall a classical result for convex functions (see \cite[Proposition 2.31]{AFP} and \cite{Pollard}, Appendix).

\begin{lemma}\label{lem:supconv}
Let $g:\R\rightarrow\R$ be a convex function. Then, there exist families $(a_j)_{j\in\mathbb{N}}$
and $(b_j)_{j\in\mathbb{N}}$ of real numbers such that
\[
g(r)=\sup_j\{\, a_j r+b_j\,\}\,.
\]
Moreover as $r\to +\infty$ $$\frac{g(r)-g(0)}{r} \nearrow  \sup_j \{a_j\}.$$
\end{lemma}

\begin{remark}
In Lemma \ref{lem:supconv}, one can select the supremum of all affine functions that equal $g$ at all rational numbers and with slope equal to or in between its left and right derivatives there. When $g$ is $\mathcal C^1$ these are just the tangents of $g$ at the rationals.
\end{remark}

We now introduce the main object we need in order to identify the relaxation of our functional $\F$.

\begin{definition}\label{def:psibar}
Let $g:\R\rightarrow[0,\infty)$ be as in Definition \ref{def:energy}.
Let $(a_j)_{j\in\mathbb{N}}$ and $(b_j)_{j\in\mathbb{N}}$ be the two families given by the previous lemma.
We define
\[
\ov{g}(r) =  \sup\{\, a_j r + b_j \,:\, j\in\mathbb{N}\,,\, b_j\geq0 \,\}\,,
\]
\end{definition}

\begin{remark}
\label{ajpos}
Notice that since $\psi$ is increasing, we have that $a_j \geq 0$ for all $j \in \N$.
\end{remark}

\begin{proposition}\label{prop:convsub}
Let $\psi:\R\rightarrow[0,\infty)$ be as in Definition \ref{def:energy}.
Then $\ov{\psi}$ is the convex subadditive envelope of $\psi$.
\end{proposition}

We divide the proof of the above proposition in a sequence of lemmas.

\begin{lemma}
Let $\psi$ and $\ov{\psi}$ be as in Definition \ref{def:psibar}. Then $\ov{\psi}$ is convex and subadditive.
\end{lemma}
\begin{proof}
As a supremum of affine functions, $\ov{\psi}$ is convex.
Further, for all $\varepsilon > 0$ there exists $j\in\mathbb{N}$ such that
\begin{align*}
\ov{\psi}(r+s) &\leq a_j(r+s) + b_j + \varepsilon \\
			&\leq a_j r + b_j + a_j s + b_j + \varepsilon \quad \text{(since } b_j \geq 0) \\
			&\leq \ov{\psi}(r) + \ov{\psi}(s) + \varepsilon
\end{align*}
The arbitrariness of $\varepsilon > 0$ leads to the subadditivity.
\end{proof}

\begin{lemma}\label{stick}
Let $\psi, \ov{\psi}$ be as above. Let
\[
y=\psi'(r)r+b(r)
\]
be the equation of the tangent line to the graph of $\psi$ at the point $\left(r,\psi(r)\right)$.
Define
\[
s_0:=\sup\{\, r\in[0,\infty) \,:\, b(r) \geq 0 \,\}\,.
\]
Then $\psi\equiv\ov{\psi}$ in $[0,s_0]$, and $\ov{\psi}$ is linear on $[s_0,\infty)$ (with eventually $s_0 = +\infty$).
\end{lemma}

\begin{proof}
Notice that $b(0)=\psi(0)>0$ and that, since $\psi$ is $\mathcal C^1$ and convex, $b$ is non-increasing and continuous.
Thus, we have two cases: either $b(r)\geq 0$ for all $r\in(0,\infty)$, and in that case $\psi=\ov{\psi}$ in all $[0,\infty)$,
or there exists a point $r\in(0,\infty)$ such that $b(r) < 0$. In the latter, by continuity and monotonicity of $b$, we have that
\begin{equation}\label{eqn: xnot}
s_0:= \sup\{\, r\in[0,\infty) \,:\, b(r)\geq 0 \,\}
\end{equation}
is a well defined number in $(0,+\infty)$. Since, by definition,
\[
\ov{\psi} = \sup\{\, a_j r + b_j \,:\, j\in\mathbb{N}\,,\, b_j\geq0 \,\}\,,
\]
it is now clear, using Lemma \ref{lem:supconv}, that $\psi$ and $\ov{\psi}$ coincide on $[0,s_0]$. Moreover, from the above we get that a maximizing sequence in the definition of $\ov{\psi}(r)$ when $r>s_0$ satisfies $b_j \to 0$, thus $\ov{\psi}$ is a linear extension of $\psi$ past $s_0$.
\end{proof}

\begin{proof}[Proof of Proposition \ref{prop:convsub}]
Call $R$ the convex subadditive envelope of $\psi$. In the case $\ov{\psi} = \psi$ we have $\ov{\psi} = \psi = R$ so there is nothing to prove.
Assume that $\ov{\psi} = \psi$ only on some $[0, s_0]$. Assume, by contradiction, that there exists $r_* \geq s_0$ such that $\ov{\psi}(r_*) < R(r_*) \leq \psi(r_*)$, and still call (by abuse) $r_* \geq s_0$ the infimum of such points.
Then we have
\[
\psi(r_*) = \ov{\psi}(r_*),  \quad \psi'(r_*) \geq \ov{\psi}(r_*)=:a\,,
\]
and since $r_* \geq s_0$,
\[
\psi(r_*)-\psi'(r_*)r_* \leq \ov{\psi}(r_*) - ar_* = 0\,.
\]
By Lemmas \ref{cvxsubaddparacd} and \ref{noninc}, one has 
$$\psi(r)-\psi'(r)r \equiv 0$$
for all $r \geq r_*$, i.e. $\psi \equiv \ov{\psi}$ there, which contradicts our assumption.
\end{proof}

\begin{remark}\label{recipe}
The above result is still valid even if $\psi$ is not $\mathcal C^1$, by the same arguments using the right-derivatives.
Since $\psi$ is $\mathcal C^1$, we can give another characterization of $\ov{\psi}$ through the \textit{parabolicity condition}
\begin{equation*}
\psi(r) - \psi'(r)r \geq 0.
\end{equation*}
From \eqref{eqn: xnot} we can infer that
	\begin{equation}
	\ov{\psi}(r) = \left\{ 
	\begin{array}{ll}
	\psi(r) & \text{if $r\in [0,s_0)$},\\
	\psi'(s_0)r & \text{if $r\in[s_0,+\infty)$},
	\end{array}
	\right.
	\end{equation}
where 
	\[
	s_0:=\sup\{r\in \R_+ \ | \ \psi(r)-\psi'(r) r  > 0\}\,.
	\]
\end{remark}

\section{Mass preserving curves with prescribed (tangential) initial velocity}\label{cstmasscurve} 

Let $(E,u)$ as in Definition \ref{def:av} and assume also that $\psi,\psi', u, H_{\pa E}$ satisfies hypothesis (H). We show here that the set $\text{Ad}(E,u)$, as used in the proof of Proposition \ref{prop:EL} and \ref{prop:charregcrit}, in this context plays the role of the \textit{tangent} space at the point $(E,u)$ to the \textit{manifold} $\text{Cl}(m)$. In particular, for any couple $(v,w)\in \text{Ad}(E,u)$, we build a curve $((E_t,u_t))_{|t|<\e} \in\text{Cl}(m)$ such that
	\begin{align*}
	\frac{d}{d t} \F (E_t,u_t) \Big{|}_{t=0}=\int_{\pa E} [\psi'(u(x)) w(x) + \psi(u(x)) v(x) H(x) ]\d \H^{n-1}(x).
	\end{align*}
We proceed as follows. Let $(v,w)\in \text{Ad}(E,u)$, consider the diffeomorphism $\Phi_{t}:\pa E \rightarrow \R^n$ defined as
	\[
	\Phi_{t}(x):=x+t v(x)\nu_E(x)
	\]
and consider its extension on $\R^n$ through a cut off $\varphi$ as in Remark \ref{rem:just}.  Fix $\xi \in C_b^1(\pa E)$such that
	\[
	\int_{\pa E} \xi(x) \d \H^{n-1}(x)>0
	\]
and for $t,s\in (-\e,\e)$ define the curve 
	\[
	E_t:=\Phi_t(E); \ \ \ \ u_{t,s}(x):=u(\Phi_t^{-1}(x))+t w(\Phi_t^{-1}(x))+s\xi(\Phi_t^{-1}(x)) \ \ \ \text{on $\pa E_t$}.
	\]
Define the $C^1$ function
	\[
	\phi(t,s):=|E_t|+\int_{\pa E_t} u_{t,s}(x) \d \H^{n-1}(x)
	\]
and notice that 
	\begin{align}
	\phi(0,0)&=m \label{eqn:a}\\
	 \frac{\pa \phi}{\pa t}(0,0)&=0, \label{eqn:b}\\
	\frac{\pa \phi}{\pa s}(0,0)&>0 . \label{eqn:c}
	\end{align}
Indeed, relations \eqref{eqn:a} and \eqref{eqn:b} follow respectively by  the construction of $E_t$ and the same computation explained in Remark \ref{rem:just}. For \eqref{eqn:c} instead, we immediately see that it is just a consequence of our choice of $\xi$ thanks to 
	\begin{align*}
	\frac{\pa \phi }{\pa s} (0,0)&= \Big{|}_{(t,s)=0}\int_{\pa E_t} u_{t,s}(x) \d \H^{n-1}(x)\\
	&=\frac{\pa }{\pa s} \Big{|}_{(t,s)=0} \int_{\pa E} [u(x) + t w(x)+s\xi(x)]J^{\pa E} \Phi_t(x) \d \H^{n-1}(x)\\
	&= \int_{\pa E} \xi(x)  \d \H^{n-1}(x)>0.
	\end{align*}
The implicit function theorem applied to the function $\phi$ now guarantees that (up to further decrease $\e$) we can find a curve $\gamma:(-\e,\e)\rightarrow (-\e,\e)$ such that $\gamma(0)=0$ and 
	\begin{equation}\label{constrained curves}
	\phi(t,\gamma(t))=m.
	\end{equation}
This means that $\left(E_t, u_{t,\gamma(t)}\right)\in \text{Cl}(m)$ for all $t\in (-\e,\e)$. Moreover  by differentiating \eqref{constrained curves} and thanks to \eqref{eqn:b},\eqref{eqn:c} we also obtain $\dot{\gamma}(0)=0$. Hence 
	\begin{align*}
	\frac{d}{dt}\Big{|}_{t=0} \F \left(E_t, u_{t,\gamma(t)}\right) &=\frac{d}{dt}\Big{|}_{t=0}\int_{\pa E} \psi(u(x)+t(w(x)+\gamma(t) \xi(x) ) J^{\pa E}\Phi_t(x) \d \H^{n-1}(x)\\
	&= \int_{\pa E} [ \psi'(u(x))w(x)+\psi(u) v(x)H(x) ]  \d \H^{n-1}(x).
	\end{align*}
In particular, in order to compute the constrained first variation, we can restrict ourselves to any generic curve with prescribed initial velocity $(v,w)\in \text{Ad}(E,u)$.

\section{Compactness}

\begin{theorem}\label{app:comp}
Let $((E_k, \mu_k))_{k\in \N} \subset \mathfrak S$ with $E_k\Subset B_R$, for some $R>0$, be a sequence such that
	\[
	\sup_{k\in \N} \ov{\F}(E_k,\mu_k)<+\infty\,,
	\quad\text{ or }\quad
	\sup_{k\in \N} \F(E_k,\mu_k)<+\infty\,.
	\]
Then, up to a subsequence it holds $(E_k,\mu_k)\rightarrow (E,\mu)$ for some $(E,\mu) \in \mathfrak S$.
\end{theorem}
\begin{proof}
From the fact that $\ov{\psi}(r)\geq\psi(0)+ \Theta r$ we gain
\[
\psi(0)P(E_k)+\Theta\mu_k(\R^n)\leq \ov{\F}(E_k,\mu_k)\leq\F(R_k,\mu_k)\,,
\]
and, in turn
\[
\sup_{k\in \N} P(E_k) <+\infty\,,\quad\quad
\sup_{k\in \N} \mu_k(\R^n) <+\infty\,.
\]
Thanks to the compactness theorem for sets of finite perimeter (see \cite[Theorem 12.26]{Maggi}) and from the weak*-compactness for finite Radon measures (see Lemma \eqref{weak*comp}) we conclude.
\end{proof}

\section{Proof of Theorem \ref{thm: main theorem wriggling}}

This theorem is, of course, completely independent from our energy functional setting and could be proven by simplified versions of Proposition \ref{prop:density} and Lemma \ref{lem:wrigg}. However, for the sake of shortness, we prefer to derive it directly as a consequence of our construction of recovery sequences.  Pick $\psi(s) = 1+s^2/2$, for which
$s_0=\sqrt{2}$. Now choose $$\mu =s_0(1+f) \H^{n-1}\restr\pared E.$$
Since $u=s_0(1+f) \geq s_0$, from \eqref{locaverage} we get that $u_i^k \geq s_0$ in the proof of Proposition \ref{propo ac} since the $u_i^k$ are averages of $u$. Thus, the $E_k$ will always be wriggled locally by a factor $1+f$ and we will always have $$\mu_k = s_0 \H^{n-1}\restr\pa E_k.$$ More precisely the recovery sequence from Theorem \ref{thm:relax} (ii) satisfies
\begin{equation*}
E_k \to E \text{ in }L^1\,, \qquad \quad s_0|D\ca E_k| \wt s_0(1+f)|D\ca E|.
\end{equation*}
Having \eqref{borelconv} in mind, this concludes.

\section{Further geometric constraints}
To take into account additional physical constraints, for instance when depositing adatoms respectively on a flat surface or in a cylindrical box, one can replace everywhere in the above analysis the perimeter $P(E)$ with the relative perimeter $P(E;A)$ where $A$ is an open half-space or an open cylinder. In the statements about critical points or minimizers, balls can then be replaced by the suitable isoperimetric set: half-balls in the case of a half-space, balls in the corners in the case of a cylinder for small masses, and flat graphs for large enough masses.

\subsection*{Acknowledgement}
We would like to thank Irene Fonseca and Giovanni Leoni for bringing this problem to our attention and for fruitful discussions. 
We also thanks the Center for Nonlinear Analysis at Carnegie Mellon University for its support
during the preparation of the manuscript.

Marco Caroccia  was supported by the Fundação para a Ciência e a Tecnologia (Portuguese Foundation for Science and Technology) through the Carnegie Mellon\textbackslash Portugal Program under Grant 18316.1.5004440. Riccardo Cristoferi was supported by the National Science Foundation under Grant No. DMS-1411646.  Laurent Dietrich was supported by the National Science Foundation under the PIRE Grant No. OISE-0967140.



 
  \bibliographystyle{amsplain}
  \bibliography{refs}

\end{document}